\documentclass{article}

\usepackage{arxiv}

\usepackage[utf8]{inputenc} 
\usepackage[T1]{fontenc}    
\usepackage{hyperref}       
\usepackage{url}            
\usepackage{booktabs}       
\usepackage{amsfonts}       
\usepackage{nicefrac}       
\usepackage{microtype}      
\usepackage{lipsum}
\usepackage{graphicx}

\usepackage{mathtools}              
\usepackage{amssymb}                
\usepackage{amsthm}                 
\usepackage{amscd}                  
\usepackage{amstext}                
\usepackage{amsfonts}               
\usepackage{icomma}                 
\usepackage{enumitem}               
\usepackage{array}                  
\usepackage[matrix,arrow,curve]{xy} 
\usepackage{appendix}

\usepackage[utf8]{inputenc} 
\usepackage[T1]{fontenc}    
\usepackage{mathtools}              
\usepackage{amssymb}                
\usepackage{amstext}                
\usepackage{amsfonts}               
\usepackage{icomma}                 
\usepackage{enumitem}               
\usepackage{array}                  
\usepackage[matrix,arrow,curve]{xy} 
\usepackage[ruled,vlined,linesnumbered]{algorithm2e}            
\usepackage{hyperref}       
\usepackage{nicefrac}       
\usepackage{multicol}
\newtheorem{theorem}{Theorem}

\newtheorem{prop}{Proposition}
\newtheorem{corollary}{Corollary}
\newtheorem{remark}{Remark}

\newcommand{\W}{\mathcal{W}}
\newcommand{\Hilb}{\mathcal{H}}
\newcommand{\E}{\mathbb{E}}
\newcommand{\argmin}{\operatornamewithlimits{argmin}}
\newcommand{\argmax}{\operatornamewithlimits{argmax}}
\newcommand{\diag}{\operatorname{diag}}
\newcommand{\Supp}{\operatorname{Supp}}
\newcommand{\Prob}{\operatorname{Pr}}
\newcommand{\Mat}{\operatorname{Mat}}
\newcommand{\new}{}
\begin{document}

\title{Stochastic Saddle-Point Optimization for the Wasserstein Barycenter Problem
\thanks{
The work of A. Gasnikov in Section \ref{sec:saddle_point_repr} was partially supported by the Ministry of Science and Higher Education of the Russian Federation (Goszadaniye) 075-00337-20-03, project no. 0714-2020-0005.
The work of P. Dvurechensly in Section \ref{sec:algorithms} was funded by Russian Science Foundation (project 18-71-10108).
The work of D. Tiapkin was prepared within the framework of the HSE University Basic Research Program.
}}


\author{Daniil Tiapkin \\
 HSE University, Moscow, Russia \\
              \texttt{dtyapkin@hse.ru}  
\And
        Alexander Gasnikov\\ 
        Moscow Institute of Physics and Technology, Moscow, Russia \\
              Institute for Information Transmission Problems, Moscow, Russia \\
              HSE University, Moscow, Russia\\
             Weierstrass Institute for Applied Analysis and Stochastics, Berlin, Germany\\
              \texttt{gasnikov@yandex.ru}
              \And 
        Pavel Dvurechensky\\
                Weierstrass Institute for Applied Analysis and Stochastics, Berlin, Germany \\
                Institute for Information Transmission Problems, Moscow, Russia \\
                \texttt{pavel.dvurechensky@wias-berlin.de}
}

\maketitle

\begin{abstract}
We consider the population Wasserstein barycenter problem for random probability measures supported on a finite set of points and generated by an online stream of data. This leads to a complicated stochastic optimization problem where the objective is given as an expectation of a function given as a solution to a random optimization problem. We employ the structure of the problem and obtain a convex-concave stochastic saddle-point reformulation of this problem. 
In the setting when the distribution of random probability measures is discrete, we propose a stochastic optimization algorithm and estimate its complexity.  
The second result, based on kernel methods, extends the previous one to the arbitrary distribution of random probability measures. Moreover, this new algorithm has a total complexity better than the Stochastic Approximation approach combined with the Sinkhorn algorithm in many cases. We also illustrate our developments by a series of numerical experiments.
\end{abstract}
\label{intro}
In this paper we consider stochastic optimization problems, which arise in computational optimal transport when the goal is to estimate population Wasserstein barycenter \cite{agueh2011barycenters} (or Fr\'echet mean w.r.t. Wasserstein distance) of a probability distribution on the Wasserstein space of probability measures. Wasserstein barycenter problem has recently attracted a lot of attention from the machine learning, statistics, and optimization community, see \cite{peyre2018optimaltransport,dvinskikh2020stochastic} and references therein. From the computational point of view, approximating population Wasserstein barycenter is a challenging optimization problem since it contains several layers of complications. The first layer requires to define for $p\geq 1$ the $p$-Wasserstein distance $\W_p(r,c)$ between two probability measures $r,c$, which in this paper we assume to be discrete probability distributions that belong to the standard simplex $\Delta^n$. $\W_p(r,c)$ is defined as an optimal value in an optimal transport problem, where the goal is to find a transport plan between vectors $r$,$c$ such that the total transportation cost is minimal (see \eqref{eq:ot_problem} for a formal definition). The next layer is that the measure $c$ is assumed to be random with some probability distribution $P_c$ on $\Delta^n$ and the population barycenter $r^*$ minimizes the expected Wasserstein distance, i.e. $r^*=\argmin_{r \in \Delta^n} \E_c \W_p^p(r, c)$.
This is a stochastic optimization problem, where the objective is defined as a solution to some other random optimization problem. The third layer of complication is that the dimension $n$ can be very large, e.g. $10^6$. To motivate this setting we refer to \cite{boissard2015distribution}, where the authors consider a template estimation problem from a random sample of its transformations. A particular example of such a template can be images, considered in \cite{cuturi2014fast}, where empirically the Wasserstein barycenter gives the best quality of the template image reconstruction. For a $10^3$ by $10^3$ image, we obtain that $n=10^6$ and the dimension of the transportation plan is then $n^2=10^{12}$ leading to a huge-scale optimization problem.

Similar to other stochastic optimization problems, see e.g. \cite{nemirovski2009robust}, there are two approaches to the population Wasserstein barycenter problem, which we refer to as offline (also known as Sample Average Approximation, SAA) and online (also known as Stochastic Approximation, SA). The offline approach assumes that a number of random probability measures $c_i$, $i=1,...,m$ is sampled in advance and all the measures are stored in memory. Then the population Wasserstein barycenter problem is approximated by the empirical Wasserstein barycenter problem $\hat{r}=\argmin_{r \in \Delta^n} \frac{1}{m} \sum_{i=1}^m\W_p^p(r, c_i)$, i.e. the expectation is approximated by the finite-sum. Finally, it is assumed that an optimization algorithm has access to all of the measures $c_i$, $i=1,...,m$ immediately, probably in parallel, and the goal is to minimize this finite-sum. 
So far, nearly all algorithms in the literature for approximating Wasserstein barycenter are designed for this finite-sum problem, to name a few  \cite{cuturi2014fast,benamou2015iterative,staib2017parallel,claici2018stochastic,dvurechensky2018decentralize,uribe2018distributed,kroshnin2019complexity,lin2020revisiting,dvinskikh2020improved,heinemann2020randomised,krawtschenko2020distributed}. Additionally, in this setting, it is possible to use a big arsenal of modern decentralized distributed optimization algorithms, see for example, \cite{dvurechensky2018decentralize,uribe2018distributed,scaman2017optimal,dvinskikh2019primal,krawtschenko2020distributed,gorbunov2019optimal,hendrikx2020statistically,hendrikx2020optimal,gorbunov2020recent,rogozin2021accelerated,rogozin2021decentralized,dvurechensky2021hyperfast,beznosikov2021decentralized}, where, in particular, such algorithms are developed with best-known complexity bounds for general finite-sum optimization with cheap dual stochastic oracle, which turns out to be available for the Wasserstein distance.

Despite all these developments, the offline approach does not consider clearly the question of how well the solution to the finite-sum problem approximates the target population Wasserstein barycenter which is a solution to a stochastic optimization problem. Moreover, in practice it may not be possible to generate all the measures in advance, the storage could be problematic, and the data consisting of random measures can appear as an online stream.
Thus, in this paper, we focus on the less studied online setting with all measures unavailable in advance.
Online estimation techniques have been very recently developed in \cite{mensch2020online} for Wasserstein distance.
Online setting for population Wasserstein barycenter has started recently to attract interest in the statistical community \cite{kroshnin2019statistical,chewi2020gradient} where population Wasserstein barycenters are studied in a different from ours setting of Bures-Wasserstein space of Gaussian probability measures, and in optimization community \cite{dvinskikh2020stochastic}, where online approach is compared to offline. 

The approach of the latter paper is based on a widespread idea of entropy regularized optimal transport \cite{cuturi2013sinkhorn,peyre2018optimaltransport}. One of the main negative sides of the entropic regularization is that if the goal is to find an $\varepsilon$-approximation to the original non-regularized Wasserstein distance or barycenter, the regularization parameter has to be chosen proportional to $\varepsilon$ \cite{dvurechensky2018computational,lin2019efficient,kroshnin2019complexity,lin2020revisiting,dvurechensky2020stable}. This leads to numerical instability \cite{cuturi2013sinkhorn,stonyakin2019gradient} in the Sinkhorn's algorithm for Wasserstein distance and Sinkhorn-type algorithms for Wasserstein barycenters \cite{benamou2015iterative}.
Alternative methods, e.g. based on accelerated gradient method \cite{chernov2016fast,dvurechensky2016primal-dual,dvurechensky2018computational,anikin2017dual,guminov2019accelerated,lin2019efficient,lin2020revisiting,nesterov2020primal-dual,dvurechensky2020stable,guminov2021combination}, similarly to Sinkhorn's algorithm require the regularization parameter to be small. 
Moreover, entropic regularization leads to a blurred reconstruction of the original template \cite{cuturi2014fast}.

Thus, in this paper, we focus on the following question and give a positive answer to it.
{\textit{Is it possible to find a population Wasserstein barycenter of a set of discrete probability measures in the online setting and without the use of entropic regularization and the calculation of (regularized) Wasserstein distance or its (sub)gradients?}} 


To propose such an online approach that does not require calculating Wasserstein distance and its subgradient, we use a saddle-point representation for the Wasserstein barycenter problem. Since the population Wasserstein barycenter problem is a convex stochastic programming problem, we obtain partially stochastic convex-concave saddle-point representation. Then we adapt Stochastic Mirror Descent \cite{nemirovsky1983problem,nemirovski2009robust,bayandina2018mirror} for this, partially infinite-dimensional, saddle-point problem. Finally, by using a reproducing kernel Hilbert space trick inspired by \cite{genevay2016stochastic}, we demonstrate how to make our approach practical and obtain finite complexity bounds. 
We compare the proposed methods with the online approach of \cite{dvinskikh2020stochastic} and obtain a regime in which our complexity bounds are better. We also illustrate numerically the regime in which our methods have better performance than the online algorithm in \cite{dvinskikh2020stochastic}.

\paragraph{Notation.} 
We define \(\Mat_{n \times m}(X)\) as a space of all matrices of size \(n \times m\) with entries from the set \(X\).
We denote by \([n]\) an \(n\)-element set \(\{1, 2, \ldots, n\}\), and \(e^v, e^A, \log(v), \log(A), \sqrt{v}, \sqrt{A}\) for \(v \in \mathbb{R}^n, A \in \Mat_{n \times n}(\mathbb{R})\) as the element-wise exponent, logarithm and square root respectively.
Also we define \(\Delta^n\) to be \(n\)-dimensional probability simplex \(\Delta^n  = \{ (s_1, \ldots, s_n) \mid \forall i \in [n] : s_i \geq 0, \ \ \sum_{i = 1}^n s_i = 1\}\).
\(\mathbf{1}_n\) is \(n\)-dimensional vector consisting of ones. If the dimension is clear from the context, the subscript is omitted. As \(e_i\) we denote the \(i\)-th coordinate vector. The probability measure induced by a random variable \(\xi\) will be denoted by \(P_\xi\). 
For a matrix \(M \in \Mat_{n \times m}(X)\) we denote by \(M_{(i)}\) the \(i\)-th row of this matrix.

\section{Preliminaries}
\label{sec:background}

\subsection{Background on optimal transport}
\label{sec:ot_background}
In this section, following \cite{peyre2018optimaltransport}, we recall some basic definitions related to optimal transport. Since we deal only with discrete measures, we consider only the discrete-discrete optimal transport problem. 

\paragraph{Optimal transport (OT) problem.}
    For a fixed non-negative matrix \(C\) and two discrete probability measures \(r,c \in \Delta^n\)  with \(n\)-element support define \textit{a transportation cost} between measures \(r\) and \(c\) associated with the cost matrix \(C\) as a solution to the following optimization problem
    \begin{equation}\label{eq:ot_problem}
        L_{C}(r,c) = \min_{X \in \mathcal{U}(r,c)} \langle C, X \rangle,
    \end{equation}
    where \(X\) is called \textit{a transport plan} and \(\mathcal{U}(r,c)\) is \textit{a transport polytope}, defined as
    \(
        \mathcal{U}(r,c) = \{ X \in \Mat_{n \times n}(\mathbb{R}_+) \mid X \mathbf{1} = r, \ X^T \mathbf{1} = c \}
    \),
    and $\langle \cdot, \cdot \rangle$ is the Frobenius dot-product of two matrices.

If \(r\) and \(c\) are probability measures onto discrete \(n\)-element metric space \((\mathcal{M}, d)\), the \textit{\(p\)-Wasserstein distance} between these two measures is defined as the \(p\)-th root of the transportation cost associated with the matrix \(D^p_{i,j} = d(x_i, x_j)^{p}\), where \(x_i\) and \(x_j\) are elements of \(\mathcal{M}\). Formally,
\(
    \W_p(r,c) = \left( L_{D^p}(r,c) \right)^{1/p}.
\)
With a slight abuse of notation, we will refer to \(L_C(r,c)\) for an arbitrary non-negative cost matrix \(C\) as a Wasserstein distance too, and denote it by \(\W(r,c)\) when matrix \(C\) is fixed.

\paragraph{Dual problem.} The linear program in the definition of the transportation cost can be reformulated using so-called Kantorovich duality in two ways: using some reformulation of results from \cite{peyre2018optimaltransport} gives:
\begin{equation}\label{eq:dual_simple}
    L_{C}(r,c) = \max_{\substack{\lambda, \mu \in \mathbb{R}^n\\-C_{i,j}-\lambda_i-\mu_j \leq 0}} - \langle \lambda, r \rangle - \langle \mu, c \rangle,
\end{equation}
and an equivalent formulation is
\begin{equation}\label{eq:dual_problem}
    L_{C}(r,c) = \max_{\mu \in \mathbb{R}^n} - \langle \lambda^*(\mu, C), r \rangle - \langle \mu, c \rangle,
\end{equation}
where \(\lambda^* \colon \mathbb{R}^n \times \Mat_{n \times n}(\mathbb{R}) \to \mathbb{R}^n\) is defined element-wise: 
\begin{equation}\label{eq:lambda_def}
    \lambda_i^*(\mu, C) = \max_{j \in [n]} ( -C_{i,j} - \mu_j).
\end{equation}

\paragraph{Barycenter definition.}
Suppose that we have a random variable \(\xi \colon \Omega_\xi \to \Delta^n\) on probability simplex, or, equivalently, on the space of probability measures onto \(\mathcal{M}\). Then we can define a \(p\)-Wasserstein (population) barycenter w.r.t. \(\xi\) as the solution to the following optimization problem:
\begin{equation}\label{eq:barycenter_canonical}
    r_* = \argmin_{r \in \Delta^n} \E \left[ \W_p^p(r, \xi) \right]. 
\end{equation}
However, we are interested in a more general situation defined by using an arbitrary non-negative cost matrix \(C\):
\[
    r_* = \argmin_{r \in \Delta^n} \E\left[L_C(r, \xi)\right] = \argmin_{r \in \Delta^n} \E\left[\W(r, \xi)\right].
\]

\subsection{Notation for saddle-point problems}
\label{not_saddle_point}

In this section, we give a necessary background on saddle-point problems. For a more comprehensive description we refer to \cite{bubeck2014convex,nemirovski2009robust}.
The following convex-concave saddle point problem is of relevance to us:
\[
    \min_{x \in \mathcal{X}} \max_{y \in \mathcal{Y}} F(x,y).
\]
The typical way to evaluate the quality of the algorithm that outputs the pair \((\widetilde x,\widetilde y)\) is to use the so-called duality gap:
\[
    \max_{y \in \mathcal{Y}} F(\widetilde x, y) - \min_{x \in \mathcal{X}} F(x, \widetilde{y}) \leq \varepsilon.
\]

For the stochastic setting, where \((\widetilde x, \widetilde y)\) are random, we are mostly interested in an approximate solution with large probability:
\begin{equation}\label{eq:confidence_region_criteria}
    \Prob\left[\max_{y \in \mathcal{Y}} F(\widetilde x, y) - \min_{x \in \mathcal{X}} F(x, \widetilde{y}) \geq \varepsilon\right] \leq \sigma.
\end{equation}
We refer to $\varepsilon$ as accuracy or precision and to $\sigma$ as confidence level.
Another useful and well-known criteria is small expectation of the duality gap:
\begin{equation}\label{eq:expectation_criteria}
    \E\left[ \max_{y \in \mathcal{Y}} F(\widetilde x, y) - \min_{x \in \mathcal{X}} F(x, \widetilde{y}) \right] \leq \varepsilon.
\end{equation}

\section{Saddle-point representation}
\label{sec:saddle_point_repr}
Using the dual reformulation of the optimal transport problem and the definition of the Wasserstein barycenter, we obtain the following problem:
\[
    \min_{r \in \Delta^n} \E\left[ \W(r,\xi) \right] = \min_{r \in \Delta^n} \E \left[ \max_{\mu \in \mathbb{R}^n} - \langle \lambda^*(\mu, C), r \rangle - \langle \mu, \xi \rangle \right].
\]

For this problem, we apply Theorem 14.60 from \cite{rockafellar2009va} to the space of all \(P_\xi\)-measurable functions \(\mathcal{F}\). Clearly, it is decomposable. It is also clear that the function under maximum is a normal integrand and finite, since the barycenter is well-defined. Thus, we have the next equality:
\begin{equation}\label{eq:barycenter_basic}
    \min_{r \in \Delta^n} \E\left[ \mathcal{W}(r, \xi) \right] = \min_{r \in \Delta^n} \sup_{f_\mu \in \mathcal{F}} \E\left[-\langle \lambda^*(f_\mu(\xi), C), r \rangle - \langle f_\mu(\xi), \xi \rangle \right].
\end{equation}
Here we call \(f_\mu \colon \Delta^n \to R^n\) the function that assigns a value of Kantorovich potential \(\mu\) for every point in \(\Delta^n\). In other words, we are trying to find potentials simultaneously for every possible measure. We will refer to such functions as potential functions. The Wasserstein barycenter is a solution to this stochastic saddle-point problem. 

\subsection{Bounds for the dual variables}
\label{sec:dual_var_bound}

To derive theoretical guarantees of the proposed algorithms, we need to construct a bound on the optimal variable value in the dual problem \eqref{eq:dual_problem}. We obtain it by using properties of the entropy-regularized optimal transport problem and some additional assumptions on the barycenter \(r_*\).
The core idea is to use a known bound from \cite{dvurechensky2018computational,lin2019efficient} on the optimal variable in the regularized case and transfer it into the non-regularized case by going to the limit when the regularization parameter goes to zero.

\paragraph{Entropic regularization.} Firstly, let us define the entropy regularized OT problem:
\[
    L_{C}^\gamma(r,c) = \min_{X \in \mathcal{U}(r,c)} \langle C, X \rangle - \gamma H(X),
\]
where \(H(X) = - \sum_{i,j=1}^n X_{i,j} \log X_{i,j}\) is the entropic regularizer. One can rewrite the problem in the equivalent form using Lagrange multipliers $\lambda,\mu$:
\begin{equation}\label{eq:dual_simple_regularized}
    L_{C}^\gamma(r,c) = \max_{\lambda, \mu \in \mathbb{R}^n}  -\langle \lambda, r \rangle - \langle \mu, c \rangle - \gamma \sum_{i,j=1}^n \exp\left(\frac{-C_{i,j} - \lambda_i - \mu_j}{\gamma}  - 1\right).
\end{equation}
Denote by \(\lambda^*_\gamma, \mu^*_\gamma\) the optimal variables for \eqref{eq:dual_simple_regularized}. Our goal is to connect optimal dual variables of the regularized and non-regularized problems. The needed result is presented in the following proposition.

\begin{prop}
    \((\lambda^*_\gamma, \mu^*_\gamma)\) converges to a pair of optimal dual variables of the non-regularized problem \eqref{eq:dual_simple} as \(\gamma \to 0^+\).
\end{prop}
\begin{remark}
Optimal dual variables of the non-regularized problem are not necessarily unique, but for our purposes, it is sufficient to work with one of the solutions.
\end{remark}
\begin{proof}
    Using Theorem 7.17 and Proposition 7.30 from \cite{rockafellar2009va}, it is sufficient to prove that the function maximized in \eqref{eq:dual_simple_regularized} converges point-wise to the function maximized in \eqref{eq:dual_simple} if in \eqref{eq:dual_simple} we equivalently change the constraint to its indicator function in the objective.

    Notice that by the properties of the limit it is sufficient to prove that
    \[
        \gamma \exp\left( \frac{\alpha}{\gamma} - 1 \right) \underset{\gamma \to 0^+}{\to} \begin{cases}
            +\infty & \alpha > 0 \\
            0 & \alpha \leq 0
        \end{cases}
    \]
    as a function of \(\alpha\). 
    
    If \(\alpha \leq 0\), the claim clearly holds. If \(\alpha > 0\), then
    \[
        \frac{\alpha}{\gamma} - 1 + \log \gamma = \frac{\alpha + \gamma \log \gamma}{\gamma} - 1 \to +\infty \;\; \text{as} \;\; \gamma \to 0^+.
    \]\qed
\end{proof}

Thus, it is sufficient to provide a $\gamma$-dependent bound  for the norm of \(\lambda^*_\gamma, \mu^*_\gamma\) and take the limit as $\gamma \to 0^{+}$ to obtain a bound for a pair of optimal dual variables in the non-regularized problem.

To obtain the bounds in the regularized case, we, first, make the change of variables: \(u = -\nicefrac{\lambda}{\gamma} - \nicefrac{1}{2}, v = -\nicefrac{\mu}{\gamma} - \nicefrac{1}{2}\). Denote by \(K\) the matrix \(e^{-C/\gamma}\). Then we can rewrite \eqref{eq:dual_simple_regularized} in the following form using the notation of \cite{cuturi2013sinkhorn}:
\begin{equation}\label{eq:dual_regularized}
    L_{C}^\gamma(r,c) = \gamma \max_{v,u} \left[ \langle u, r \rangle + \langle v,c \rangle - \langle \mathbf{1},  \diag{e^u} K \diag{e^v}\rangle  \right] + 2 \gamma.
\end{equation}

\begin{remark}
    If \((u^*, v^*)\) are optimal in \eqref{eq:dual_regularized}, then \((u^* + \alpha \mathbf{1}, v^* - \alpha \mathbf{1})\) are optimal too. Hence, we can assume \(\max_i v^* = \max_{i \in [n]}(-\nicefrac{\mu_\gamma^*}{\gamma} - \nicefrac{1}{2}) = -\nicefrac{1}{2} \iff \min_{i \in [n]} (\mu_\gamma^*)_i = 0\).
\end{remark} 

Now, under assumption \(r > 0, c > 0\), it is possible to use bounds from Lemma 1 of the work \cite{dvurechensky2018computational} :
\begin{gather*}
    \max_i (u^*)_i - \min_i (u^*)_i \leq R, \\
    \max_i (v^*)_i - \min_i (v^*)_i \leq R,
\end{gather*}
where \(R := - \log( \nu \min_{i \in [n]}\{r_i, c_i\} ), \nu = e^{-\nicefrac{\Vert C\Vert_\infty}{\gamma}}\).

Returning to the original variables:
\begin{equation}\label{eq:linfty_norm_bound}
    \max_{i} (\mu_\gamma^*)_i - \min_i (\mu_\gamma^*)_i \leq \Vert C \Vert_\infty - \gamma \log \min_{i\in[n]} \{r_i, c_i\}.
\end{equation}
We know that \(\min_{i \in [n]} (\mu_\gamma^*)_i = 0 \iff \mu_\gamma^* \geq 0\). Hence, \eqref{eq:linfty_norm_bound} can be rewritten as
\[
    \Vert \mu_\gamma^* \Vert_\infty \leq \Vert C \Vert_\infty - \gamma \log \min_{i\in[n]} \{r_i, c_i\}.
\]

By tending \(\gamma\) to \(0^+\) and noting that problem \eqref{eq:dual_problem} is equivalent to problem \eqref{eq:dual_simple}, we prove the following theorem.
\begin{theorem}\label{th:dual_bound}
    Assume that $r,c >0$. Then, there exists an optimal solution \(\mu^*\) for the problem \eqref{eq:dual_problem} such that \( \Vert \mu^* \Vert_\infty \leq \Vert C \Vert_\infty \).
\end{theorem}

Additionally, there in an immediate corollary of Theorem \ref{th:dual_bound}.
\begin{corollary}\label{cor:variable_bound}
    Assume that \((r^*, f_\mu^*)\) are optimal variables in the saddle-point problem \eqref{eq:barycenter_basic} in which all $\xi >0$. Then, if \(r^* > 0\), there exists another optimal \(\hat{f_\mu}\) such that  \(\Vert \hat{f_\mu} \Vert_\infty \leq \Vert C \Vert_\infty\) \(P_\xi\)-almost surely.
\end{corollary}

Using Corollary \ref{cor:variable_bound}, we can equivalently reformulate problem \eqref{eq:barycenter_basic} in the following manner:
\begin{equation}\label{eq:barycenter}
    \min_{r \in \Delta^n} \E\left[ \mathcal{W}(r, \xi) \right] = \min_{r \in \Delta^n} \sup_{f_\mu \in \mathcal{F}^b} \E\left[-\langle \lambda^*(f_\mu(\xi), C), r \rangle - \langle f_\mu(\xi), \xi \rangle \right],
\end{equation}
where \(\mathcal{F}^b = \{ f \colon \Delta^n \to \mathbb{R}^n \mid f\) is \(P_\xi\)-measurable, \(\Vert f \Vert_\infty \leq \Vert C \Vert_\infty\) \(P_\xi\)-a.s.\(\}\). We will use this form of the problem in the further analysis.

Notice that this reformulation is valid in the case of an existence of a positive solution to the original problem. 
If all $\xi$ are positive, then the solution to the original problem is also positive. At the same time, all $\xi$ can be made positive by adding the vector of all ones multiplied by a small number, similarly to as it was done in the analysis of the Sinkhorn algorithm for Wasserstein distances in \cite{dvurechensky2018computational}. 

\section{Algorithms}
\label{sec:algorithms}

In this section, we provide algorithms for computation of population Wasserstein barycenters using the saddle-point formulation \eqref{eq:barycenter} combined with an assumption of the existence of a positive optimal \(r^*\).

\subsection{Finite support case}
\label{sec:fin_support}

First we study a simpler situation and assume that the random variable \(\xi\) has a finite support \(\Supp(\xi) = \{c_1,\ldots, c_m\}\).
Then, any function \(f_\mu \in \mathcal{F}^b\) can be represented as a matrix  of size \(m \times n\) and the value \(f_\mu(c_i)\) of this function is defined as the \(i\)-th row of this matrix. The value of this function on points outside of \(\Supp(\xi)\) is defined to be zero. The corresponding matrix is denoted by \(M(f_\mu)\). 

To make the notation simpler, we say that we have a random variable \(\zeta\) over indices \([m]\), such that \(c_{\zeta} = \xi\). Also, we have \(f_\mu(\xi) = M(f_\mu)_{(\zeta)}\) and denote the uniform distribution over \([n]\) by \(U([n])\) and the distribution associated with \(r\) by \(P(r)\). 

Hence, if we rewrite \eqref{eq:barycenter} in terms of \(\zeta\) and use the matrix notation, we obtain the following problem
\begin{equation}\label{eq:dual_finite}
    \min_{r \in \Delta^n} \E\left[ \mathcal{W}(r, \xi) \right] = \min_{r \in \Delta^n} \max_{\substack{M \in \Mat_{m \times n}(\mathbb{R}) \\ \Vert M \Vert_\infty \leq \Vert C \Vert_\infty}} \E_\zeta \left[-\langle \lambda^*(M_{(\zeta)}, C), r \rangle - \langle M_{(\zeta)}, c_\zeta \rangle \right].
\end{equation}

This non-smooth non-strongly convex-concave saddle-point problem can be solved using the mirror descent algorithm for saddle point problems. We refer to 
\cite{nemirovski2009robust,bubeck2014convex} for additional details on this basic algorithm.

\begin{theorem}
    Algorithm \ref{alg:finite_support} outputs a pair \((\widetilde r, \widetilde M)\) that satisfies for the problem \eqref{eq:dual_finite} the duality gap criteria \eqref{eq:confidence_region_criteria} with the precision \(\varepsilon\) and the confidence level \(\sigma\) in 
     \[
        N = O\left(\frac{n \max(n\log n, m) \Vert C \Vert_\infty^2}{\varepsilon^2}  \log^2\left(\dfrac{1}{\sigma}\right)  \right)
    \]
    iterations and its total complexity is \(O(nN)\).
\end{theorem}

 \begin{algorithm}\label{alg:finite_support}
        \KwData{\(N\) -- number of iterations;}
        \KwResult{\(\widetilde r\) -- approximation of barycenter;}
        \Begin{
            Set \(\widetilde{r} = r = \left( \nicefrac{1}{n}, \ldots, \nicefrac{1}{n} \right) \in \Delta^n\)\;
            Set \(M = 0_{m \times n} \in \Mat_{m \times n}(\mathbb{R})\)\;
            Set \(\alpha = 2\log n, \beta = 4mn\Vert C \Vert_\infty \)\;
            Set \(\eta = \frac{2}{\Vert C \Vert_\infty \sqrt{8n^2 \log n + 16mn} \cdot \sqrt{5N}}\)\;
            \For{k = 1 to N} {
                Sample \(t\) from \(P_\zeta\)\;
                Sample \(s\) from \(U([n])\)\;
                Sample \(q\) from \(P(r_{k-1})\)\;
                \(g_s :=  - n \cdot \max_{j \in [n]} (-C_{s,j} - M_{t,j})\)\;
                \(J_q = \argmax_{j \in [n]}(-C_{q,j} - M_{t,j}) \)\;
                \(h_t := -e_{J_q} + c_t,\)\;
                \(r_s := r_s\exp(- \alpha \cdot \eta \cdot g_s)\)\;
                \(r := r / (\sum_{i=1}^n r_i)\)\;
                \(M_{(t)} := M_{(t)} - \beta \cdot \eta \cdot h_t\)\;
                Project \(M\) onto \(\ell_\infty\) box\;
                \(\widetilde r := \frac{1}{k} r + \frac{k-1}{k} \widetilde r\)\;
            }
            
            Return \(\widetilde r\)\;
        }
        \caption{Mirror Descent for Finite Support case}
\end{algorithm}

\begin{remark}
    \(\widetilde r\) is an \(\varepsilon\)-approximation of the Wasserstein barycenter with probability at least \(1-\sigma\).
\end{remark} 

\begin{remark}
     If we assume that \(\Vert C \Vert_\infty^2\), \(\sigma\) are constants and \(m \geq n \log n\), Algorithm \ref{alg:finite_support} has complexity \(\widetilde{O}(n^2 m \cdot  \varepsilon^{-2})\). This is the same complexity as for the Iterative Bregman Projections (IBP) algorithm \cite{benamou2015iterative,kroshnin2019complexity}. In the case when \(m \leq n \log n\) it is also possible to obtain the complexity  \(O(n^2 m \cdot  \varepsilon^{-2})\) but under the weaker convergence criteria \eqref{eq:expectation_criteria} (see for details Remark \ref{rm:fin_sup_convergence_expectation} after the proof of the theorem).
\end{remark}

\begin{proof}
    We have a problem of type
    \[
        \min_{r \in \mathcal{X}} \max_{M \in \mathcal{Y}} \E \Psi(r,M,\zeta),
    \]
    where \(\mathcal{X} := \Delta^n, \mathcal{Y} := \{ M \in \Mat_{m \times n}(\mathbb{R}) \mid \Vert M \Vert_\infty \leq \Vert C \Vert_\infty\}\). In our case
    \[
        \Psi(r,M,\zeta) = -\langle \lambda^*(M_{(\zeta)}, C), r \rangle - \langle M_{(\zeta)}, c_\zeta \rangle,
    \]
    where \(\lambda^*\) is defined in \eqref{eq:lambda_def}.

    To use the mirror descent algorithm, we need to define prox-structures on spaces \(\mathcal{X}\) and \(\mathcal{Y}\).
    \begin{itemize}
        \item As a prox-structure on \(\mathcal{X}\) we choose the entropy setup with prox-function \(d_\mathcal{X}(r) = \sum_{i=1}^n r_i \log r_i\). We can easily bound the prox-diameter of \(\mathcal{X}\) as \(R^2_\mathcal{X} = \log n\). Then the Bregman projection is calculated by the exponential weighting in \(O(n)\) time. The corresponding norm is the \(\Vert \cdot \Vert_1\) norm.
        \item As a prox-structure on \(\mathcal{Y}\) we choose the Euclidean setup with prox-function \(d_\mathcal{Y}(M) = \frac{1}{2} \Vert M \Vert_F^2\). Clearly, we have \(R^2_\mathcal{Y} = m \cdot n \cdot 2 \Vert C \Vert_\infty^2\). The Bregman projection can be computed in linear time.  The corresponding norm is the \(\Vert \cdot \Vert_F\) norm.
    \end{itemize}
   Further, we endow the space \(\mathcal{Z} : =\mathcal{X} \times \mathcal{Y}\) with the prox-function \(d_{\mathcal{Z}}(r,M) = \frac{1}{2 \log n} d_{\mathcal{X}}(r) + \frac{1}{4 mn \Vert C \Vert_\infty^2} d_{\mathcal{Y}}(M) \).
    
    Next, let us define stochastic (sub)gradient oracles. For this purpose, we use a sample \(t \sim P_\zeta\) as the source of randomness. We have:
    \begin{align*}
        G(r, M, t) &= \partial_{r} \Psi(r, M, t) = -\lambda^*(M_{(t)}, C), \\
        H(r, M, t) &= \partial_{M_{(t)}} \left[-\Psi(r, M, t) \right] \\
        &= \sum_{i=1}^n r_i \partial_{M_{(t)}} \max_{j \in [n]} \left[- C_{ij} - M_{tj} \right] + c_t. 
    \end{align*}
   At each iteration oracle \(H\) gives us \(t\)-th row of the full subgradient matrix (up to the probability of \(c_t\)) and this way can be interpreted as a random coordinate oracle similar to random coordinate descent methods.
    
    However, these subgradients are computationally expensive and their calculation requires \(O(n^2)\) time. To make the iteration cheaper, we introduce additional randomization and apply a pure random coordinate technique to \(G\) and \(H\). More precisely, using that $r \in \Delta^n$, we sample index \(i \in [n]\) from probability distribution associated with \(r\) and calculate only \(i\)-th term of the sum in the definition of \(H\). Importantly, such stochastic subgradients remain unbiased stochastic approximations to the true subgradients.
    
    Then, we can write our final unbiased stochastic oracles using the subgradient calculus, namely subgradient of the maximum:
    \begin{align}\label{eq:g_finite}
        g(r, M, t, s) &= - n \cdot \max_{j \in [n]} (-C_{s,j} - M_{t,j})  \cdot e_s ; \\
        \label{eq:h_finite}
        h(r, M, t, q) &= -e_{J_q(M_{(t)}, t)} + c_t,
    \end{align}
    where multiplication by \(n\) in \(g\) is required to keep the unbiasedness property and \(J_q(M_{(t)}, t)\) is an index at each the value \(\lambda^*_q(M_{(t)},t)\) is achieved (see \eqref{eq:lambda_def}). These subgradients can be computed in linear in $n$ time because they require to compute maximum only over \(n\) values.
    
    For the main complexity result, we need to bound the dual norm of the stochastic subgradients 
     \begin{align*}
        \Vert g(r, M, t, s) \Vert_{\mathcal{X}^*} &= n \left\vert \max_{j \in [n]} (-C_{s,j} - M_{t,j})  \right\vert \leq 2 n \Vert C \Vert_\infty, \\
        \Vert h(r, M, t, q) \Vert_{\mathcal{Y}^*} &\leq \Vert c_t \Vert_2 + \Vert e_{J_q(M_{(t)}, c_t)} \Vert_2 \leq 2.
    \end{align*}
    
    Combining all the above, for the finite horizon of $N$ iterations and a constant step-size \(\eta = \frac{2}{L\sqrt{5N}}\), where \(L^2 = 8n^2\log n\Vert C \Vert_\infty^2 + 16mn\Vert C \Vert_\infty^2\), we apply Proposition 3.2 of \cite{nemirovski2009robust} and obtain the following complexity bound. After no more than
    \[
        N = \frac{5(8n^2\log n\Vert C \Vert_\infty^2 + 16mn\Vert C \Vert_\infty^2)}{\varepsilon^2}\left(8 + 2 \log \frac{2}{\sigma}\right)^2
    \]
    iterations Algorithm \ref{alg:finite_support} solves the saddle-point problem \eqref{eq:dual_finite} with an accuracy \(\varepsilon >0\) and a confidence level \(\sigma \in (0,1)\) in terms of \eqref{eq:confidence_region_criteria}.
    \qed
\end{proof}

\begin{remark}\label{rm:fin_sup_convergence_expectation}
    If we consider convergence of the duality gap in expectation (see \eqref{eq:expectation_criteria}), we can obtain better dependence on \(m\) for the number of iterations:
    \[
        N = O\left( \frac{n m \Vert C \Vert_\infty^2}{\varepsilon^2} \right)
    \]
    and for the total complexity: \(O\left(n^2 m \Vert C \Vert_\infty^2 \cdot \varepsilon^{-2}\right)\). In this case, the complexity bounds of our algorithm are the same as that of Iterative Bregman Projections (IBP) algorithm \cite{benamou2015iterative,kroshnin2019complexity} in all the regimes of \(m\) compared to $n$.
       
    This improvement is possible since for the analysis in terms of the expectation it is sufficient to bound the second moment of the stochastic subgradients rather than provide uniform bounds. The second moment of the dual norm of \(g\) can be bounded as \(\E\left[ \Vert g(r,M,t,s)\Vert_{\mathcal{X}^*}^2 \right] \leq 4 n \Vert C \Vert_\infty^2 =: B^2_\mathcal{X}\), whereas the uniform bound for the dual norm suffers from an additional \(\sqrt{n}\) factor. Moreover, in this setting, it is possible to use dynamic step-sizes to obtain a fully online algorithm \cite{nemirovski2009robust}.
\end{remark}

\subsection{Background on the infinite-dimensional optimization}
\label{sec:back_inf_dim_opt}
The  approach in the previous subsection can be viewed as a coordinate descent for the offline problem since when $\xi$ has finite support the expectation w.r.t. $\xi$ is actually a finite-sum. To deal with the purely online setting, we need to make additional assumptions and optimize our functional over a <<good>> subspace of \(\mathcal{F}_b\) (see \eqref{eq:barycenter} for the definition) where it can be done effectively. We follow the idea of using kernel spaces used in \cite{genevay2016stochastic} to approximate Wasserstein distance. Note that such kernel-based approach is widely used in other machine learning problems \cite{mohri2018foundations,steinwart2008support}. 

Firstly, following \cite{bogachev2020real}, we introduce a necessary background for the infinite-dimensional analysis. The main instrument for us is the Fr\`echet derivative, which is defined almost like the derivative in finite dimensions, except for the special requirements on the family of sets. For this special type of derivatives, we also have a chain rule, that we are going to actively use.

The main challenge is that this derivative is an element of the dual space that typically is not canonically isomorphic to our space. Hence, we cannot reproduce the standard proof for gradient-descent-like procedures. A possible solution is to restrict the considerations only to Hilbert spaces and apply the Riesz representation theorem to define steps in the direction of (sub)gradients. In this case, we also have a notion of subgradient $f'(x_0)$ at $x_0$ defined as
\[
    f(x) - f(x_0) \geq \langle f'(x_0), x - x_0 \rangle, \;\; x \in {\rm dom} f
\]
that is the most crucial part of the convergence proofs for gradient-descent-like procedures.

The only question is how to compute and store this infinite-dimensional object and how to compute derivatives  in \(f\) of linear functionals like \(\varepsilon_x(f) := f(x)\). For this purpose, we use an additional assumption that the space \(\mathcal{H}\) is so-called \textit{reproducing kernel Hilbert space}.

A Hilbert space \(\mathcal{H}\) of functions \(f \colon \mathcal{X} \to \mathbb{R}\) is called a reproducing kernel Hilbert space (RKHS) if there is a symmetric positive-defined function \(\mathcal{K} \colon \mathcal{X} \times \mathcal{X} \to \mathbb{R}\) called a kernel, such that
\begin{enumerate}
    \item \(\forall x \in \mathcal{X} : \mathcal{K}(\cdot, x) \in \mathcal{H}\);
    \item \(\forall f \in \Hilb : \langle f, \mathcal{K}(\cdot, x) \rangle_\Hilb = f(x)\);
    \item \(\forall x,y \in \mathcal{X} : \langle \mathcal{K}(\cdot, x), \mathcal{K}(\cdot, y) \rangle_\Hilb = \mathcal{K}(x,y)\).
\end{enumerate}

It could be proven \cite{mohri2018foundations} by an almost explicit construction that for each symmetric positive-defined function \(\mathcal{K}\) there exists such a Hilbert space. However, there are some kernels that are called \textit{universal} \cite{steinwart2008support}: they can uniformly approximate any continuous function if \(X\) is a compact subset of \(\mathbb{R}^d\). An example of such a kernel is the Gaussian RBF kernel \(\mathcal{K}(x,x') = \exp(-s \Vert x - x' \Vert_2^2)\). Due to the universal approximation property RKHS-assumption is in a sense without loss of generality.
 
\subsection{General Kernel Mirror Descent}
\label{sec:general_kmd}

In this subsection we apply the mirror descent algorithm to our partially infinite-dimensional problem using the kernel idea. We fix \((\mathcal{H},\mathcal{K})\) is a RKHS of functions from \(\Delta^n\) to \(\mathbb{R}\). Further, we search for an optimal \(f_\mu\) in the space \(\mathcal{H}^n\) of tuples of function from \(\mathcal{H}\) that forms a Hilbert space of functions from \(\Delta^n\) to \(\mathbb{R}^n\) with the inner product \(\langle f, g \rangle_{\Hilb^n} = \sum_{i=1}^n \langle f_i, g_i \rangle_\Hilb\), where \(f_i,g_i\) are \(i\)-th coordinate functions of \(f\) and \(g\) respectively.

This leads to the following new optimization problem derived from \eqref{eq:barycenter} using the RKHS assumption:
\begin{equation}\label{eq:barycenter_hilb}
    \min_{r \in \Delta^n} \E \mathcal{W}(r, \xi)= \min_{r \in \Delta^n} \sup_{f_\mu \in \mathcal{H}^{n}_b} \E\left[-\langle \lambda^*(f_\mu(\xi), C), r \rangle - \langle f_\mu(\xi), \xi \rangle \right],
\end{equation}
where \(\mathcal{H}^n_b = \{ f \in \mathcal{H}^n \mid \Vert f \Vert_\infty \leq \Vert C \Vert_\infty P_\xi\)-a.s.\(\}\) is an intersection of \(\mathcal{H}^n\) and \(\mathcal{F}^b\) (see \eqref{eq:barycenter}). 

Theoretical dependence on the kernel is hidden in the following two constants: \(R^2 = \sup_{f} \frac{1}{2} \Vert f \Vert^2_\Hilb\) and \(\kappa^2 = \sup_{x} \mathcal{K}(x,x)\), that fully depend on the chosen kernel.  Unfortunately, for most of non-trivial kernels the value of \(R^2\) is infinite because \(\Vert \cdot \Vert_\mathcal{H}\) and uniform norm \(\Vert \cdot \Vert_\infty\) are not equivalent in these infinite-dimensional spaces. Typically the value $R$ is used to bound the norm \(\Vert f \Vert_{\mathcal{H}^n} = \sqrt{ \sum_{i=1}^n \Vert f_i \Vert_{\mathcal{H}}^2 }\), where \(f_i\) is a coordinate function of \(f\). Since $R$ may be unbounded, we introduce to the problem an auxiliary constraint \(\Vert f \Vert_{\mathcal{H}^n} \leq \overline{R}\) for a sufficiently large \(\overline{R}\). In this case, if an optimal potential function lies in this ball, by Lemma 1 of \cite{antonakopoulos2019adaptive} our algorithm outputs an approximation to a solution of the original problem.
To guarantee that \(\overline{R}\) is sufficiently large, it is possible to use a doubling technique: multiply \(\overline{R}\) by two and restart the algorithm with the new value of \(\overline{R}\) until convergence. The total complexity will be asymptotically similar to the case if we choose \(\overline{R} =  \Vert f_\mu^* \Vert_{\Hilb^n}\).

The following theorem is the main result of this subsection.
\begin{theorem}
    For an arbitrary kernel \(\mathcal{K}\) and an arbitrary distribution \(P_\xi\) s.t. all $\xi>0$, Algorithm \ref{alg:general_kmd} (Kernel Mirror Descent) outputs with probability at least \(1-\sigma\)  an \(\varepsilon\)-approximation of the Wasserstein barycenter in terms of the duality gap \eqref{eq:confidence_region_criteria} in
    \[
        N = O\left(\frac{n \kappa^2 \overline{R}^2 + \Vert C \Vert^2_\infty \log n}{\varepsilon^2} \log^2\left(\frac{1}{\sigma} \right) \right)
    \] sample iterations and with
    \[
        O\left(\frac{n^3 \kappa^4 \overline{R}^4 + \Vert C \Vert_\infty^4 n \log^2 n}{\varepsilon^4} \log^4\left(\frac{1}{\sigma}\right)\right)
    \] total complexity.
\end{theorem}

\begin{remark}
    This algorithm is from the family of Stochastic Approximation (SA) algorithms and the most correct comparison can be obtained with other SA algorithms.
        In \cite{dvinskikh2020stochastic,dvinskikh2021decentralized} the following total complexity of SA approach based on entropy regularization with parameter $\gamma$ and Sinkhorn algorithm is shown:
    \begin{align*}
        \tilde O\biggl(& \frac{n^3 \Vert C \Vert^2_\infty}{\varepsilon^2} \min\biggl\{\frac{1}{\varepsilon}\sqrt{\frac{n \Vert C \Vert^2_\infty}{\gamma \mu}}, \\
        &\exp\left(\frac{\Vert C \Vert_\infty \log n}{\varepsilon} \right)\left(\frac{\Vert C \Vert_\infty \log n}{\varepsilon} + \log \left(\frac{\Vert C \Vert_\infty \log n}{\gamma \varepsilon^2} \right) \right)\biggl\} \biggl),
    \end{align*}
    where \(\mu\) is a constant of local strong convexity at the optimal point.
    In this approach it is crucial to choose \(\gamma = O(\varepsilon)\) to obtain an \(\varepsilon\)-approximation of the original non-regularized barycenter and, if we assume that all factors are constant except \(n\) and \(\varepsilon\), the complexity of our algorithm is better in \(n\) and worse in \(\varepsilon\). At the same time, 
    
    1) Choice \(\gamma = O(\varepsilon)\) leads to numerical instability of the Sinkhorn algorithm.
    
    2) \(\mu\) can be dependent on parameter \(\gamma\) and, due to the choice \(\gamma = O(\varepsilon)\), on \(\varepsilon\).
    
    3) \(\overline{R}\) can be dimension-dependent, at least in the case of the linear kernel. 
    Thus, it is not easy to compare our complexity bound and the bound of \cite{dvinskikh2020stochastic,dvinskikh2021decentralized}. But, our approach does not use entropy regularization and, thus, does not encounter the numerical instability of the Sinkhorn algorithm.
\end{remark}

    \begin{algorithm}\label{alg:general_kmd}
        \KwData{\(N\) -- number of iterations, \(\mathcal{K}\) -- kernel, \(\overline{R}^2\), \(\kappa^2\) -- constants associated with the kernel;}
        \KwResult{\(\widetilde r\) -- approximation of barycenter;}
        \Begin{
          Set \(\widetilde{r} = r = \left( \nicefrac{1}{n}, \ldots, \nicefrac{1}{n} \right) \in \Delta^n\)\;
          Set \(f_\mu = 0\)\;
          Set \(\alpha = 2\log n, \beta = 2n\overline{R^2}\)\;
          Set \(\eta = \frac{2}{\sqrt{8 \log n \Vert C \Vert^2_\infty + 8 n^2 \kappa^2 \overline{R}^2} \cdot \sqrt{5N}}\)\;
          \For{k = 1 to N} {
                Sample \(c^{(k)}\) from \(P_\xi\)\;
                \For{t = 1 to n} {
                    \((f_\mu(c^{(k)}))_t = \min\{1, \max\{-1, \sum\limits_{i=1}^{k-1} \beta^{(i)}_t \mathcal{K}(c^{(k)}, c^{(i)}) \}\}\)\;
                }
                \For {i = 1 to n} {
                    \(J_i = \arg\max_{j \in [n]} (-C_{i,j} - f_\mu(c^{(k)})_{j})\)\;
                    \(g_i = -\max_{j \in [n]} (-C_{i,j} - f_\mu(c^{(k)})_{j}) \)
                }
                \For {t = 1 to n} {
                \[
                    \beta^{(k)}_t = \eta \cdot \beta \cdot \left(-c^{(k)}_t +  \sum_{i=1}^n r^{(k)}_i I\{t = J_{i}(c^{(k)}))\} \right)
                    \]
                }
                \(r := r \cdot \exp(-\eta \cdot \alpha \cdot g)\)\;
                \(r : = r / (\sum_{i=1}^n r_i)\)\;
                \(\widetilde r := \frac{1}{k} r + \frac{k-1}{k} \widetilde r\)\;
          }
          Return \(\widetilde r\)\;
        }
        \caption{Kernel Mirror Descent (KMD)}
    \end{algorithm}

\begin{proof}
We are dealing with stochastic programming problem of the form
\[
        \min_{r \in \mathcal{X}} \max_{f_\mu \in \mathcal{Y}} \E \Phi(r,f_\mu,\xi),
\]
where \(\mathcal{X} := \Delta^n, \mathcal{Y} := \mathcal{H}^n_b\), and
\[
    \Phi(r,f,c) = -\langle \lambda^*(f(c), C), r \rangle - \langle f(c), c \rangle.
\]

For \(\mathcal{X}\) we use the entropy setup exactly as in the finite support case: there is no effect of the infinite dimension of \(f_\mu\). For \(\mathcal{Y}\) we define the prox-function \(d_{\mathcal{Y}}(f) = \frac{1}{2} \Vert f \Vert_{\mathcal{H}^n}^2 = \frac{1}{2} \sum_{i=1}^n \Vert f_i \Vert_{\mathcal{H}}^2\). The corresponding prox-diameter is \(\mathcal{R}^2_\mathcal{Y} = \overline{R}^2\). 
The corresponding Bregman divergence is
\[
    B_\mathcal{Y}(f,g) = \frac{1}{2} \Vert g \Vert_{\mathcal{H}^n}^2 - \frac{1}{2} \Vert f \Vert_{\mathcal{H}^n}^2 - \langle f, g-f\rangle_{\mathcal{H}^n} = \frac{1}{2} \Vert f-g \Vert_{\mathcal{H}^n}.
\]
Hence, the Bregman projection is the projection onto \(\mathcal{H}^n_b\). Finally, for the product space \(\mathcal{Z} := \mathcal{X} \times \mathcal{Y}\) we define \(d_\mathcal{Z}(r, f_\mu) = \frac{1}{2\log n}d_\mathcal{X}(r) + \frac{1}{2\overline{R}^2}d_\mathcal{Y}(f_\mu)\).

The last component of the analysis is the computation of the subgradients and finding bounds for their dual norm. For \(f_\mu\) we compute the subgradient coordinate-wise, i.e. for each fixed coordinate \(t\):
\begin{align*}
    G(r, f, c) &= \partial_{r} \Phi(r, f, t) = -\lambda^*(f(c), C), \\
    H(r, f, c)_t &= \left(\partial_{f}(-\Phi(r,f(c),c))\right)_t \\
    &= \left( \partial_{f}(-\Phi(r,\langle f, \mathcal{K}(\cdot,c) \rangle_\Hilb,c))\right)_t \\
    &= \mathcal{K}(\cdot, c) \cdot \left( \partial_{\mu}(\langle \lambda^*(\mu, C), r \rangle + \langle \mu, c \rangle)\right)_t \\
    &= \mathcal{K}(\cdot, c) \left( c - \sum_{i = 1}^n r_i I\{t = J_i(c)\} \right),
\end{align*}
where \(J_i(c)\) is one of the indices in \(\lambda^*_i\) where the maximum value is achieved. We can bound the norms of both subgradients as follows:
\begin{align*}
    \Vert G(r,f,c) \Vert_\infty &\leq 2 \Vert C \Vert_\infty \\
    \Vert H(r, f_\mu, c)_t \Vert_\Hilb &\leq 2 \Vert \mathcal{K}(\cdot, c) \Vert_\Hilb \leq 2 \kappa,
\end{align*}
where we used that \( \Vert \mathcal{K}(\cdot, c) \Vert_\Hilb = \sqrt {\langle \mathcal{K}(\cdot,c), \mathcal{K}(\cdot, c) \rangle_\Hilb} = \sqrt{\mathcal{K}(c,c)} \leq \kappa\). Aggregating the bound for each coordinate, we obtain
\[
    \Vert H(r,f_\mu,c) \Vert_{\Hilb^n}^2 \leq 4 n \kappa^2.
\]

Finally, for a fixed budget of $N$ iterations, we choose a constant step-size \(\eta = \frac{2}{L \sqrt{5N}}\), where \(L^2 = 8 \log n \Vert C \Vert_\infty^2 + 8 n \kappa^2 \overline{R}^2\) and, applying the result of \cite{nemirovski2009robust}, we obtain the iteration complexity
\[
    N = \frac{5(8 \log n \Vert C \Vert_\infty^2 + 8{ n \kappa^2 \overline{R}^2})}{\varepsilon^2}\left(8 + 2 \log \frac{2}{\sigma} \right)^2
\]
to find an approximate solution with accuracy $\varepsilon>0$ and confidence level $\sigma \in (0,1)$.

The last step is to calculate the complexity of (sub)gradient computations. Note that the actual value of \(f_\mu^{(k)}\) is used at each iteration $k$. For a step-size sequence \(\eta_k\) the following results gives an explicit formula for calculating \(f_\mu^{(k)}\) for the case \(f_\mu^{(0)} = 0\).
\begin{prop}
    If \(f_\mu^{(0)} = 0\), we have the following formulas to calculate \(f_\mu^{(k)}\):
    \begin{align}\label{eq:beta_calculation}
    &(f_\mu^{(k)})_t = \sum_{i=1}^{k} \beta^{(i)}_t \cdot \mathcal{K}(\cdot, c^{(i)}), 
    &\beta^{(k)}_t = \eta_k \cdot \left(-c^{(k)}_t +  \sum_{i=1}^n r^{(k)}_i I\{t = J_{i}(c^{(k)}))\} \right),
\end{align}
where \(\{c^{(k)}\}\) are samples from \(P_\xi\).
\end{prop}
\begin{proof}
    We proceed by induction on \(k\). The basis case \(k=0\) is clear. The induction step follows from the formula of a step of the gradient descent:
    \[
        (f_\mu^{(k)})_t = (f_\mu^{(k-1)})_t - \eta_k H(r, f_\mu^{(k-1)}, c^{(k)})_t = (f_\mu^{(k-1)})_t + \beta^{(k)}_t \mathcal{K}(\cdot, c^{(k)}).
    \]\qed
\end{proof}
    
     The proof is finished by observing that we need \(O(N)\) iterations to recalculate one coordinate function of \(f_\mu\). To calculate the subgradient w.r.t. \(r\) we need to calculate the \textit{full} vector \(f_\mu\) for a new sampled measure \(c^{(k)}\), which costs \(O(nN)\). Thus, the total complexity is \(O(nN^2)\).
    \qed
\end{proof}

We use the described kernel approach with the Gaussian RBF kernel \cite{minh2010}
\[
    \mathcal{K}(x,x') = \exp\left(-s \Vert x- x'\Vert_2^2\right),
\]
and an information-diffusion kernel for simplex \cite{lafferty2005diffusion}
\[
    \mathcal{K}(x,x') = \exp\left(-\frac{1}{t}\arccos^2\left(\langle \sqrt{x} , \sqrt{x'}\rangle\right)\right).
\]
These kernels are very similar and their main properties are also similar. Namely, for both kernels \(\kappa^2 = \sup_{x} \mathcal{K}(x,x) = 1\), and {the value of \(\overline{R}^2\) needs to be tuned}. Thus, we use it as a hyperparameter of our algorithm. Note that the values of the hyperparameters \(s\) and \(t\) need to be tuned as well.

\paragraph{Dynamic step-size.}

If we consider the convergence of the duality gap in expectation (see \eqref{eq:expectation_criteria}) instead of finding an approximate solution with large probability according to 
\eqref{eq:confidence_region_criteria}, we may use a decreasing step-size approach \cite{nemirovski2009robust} and obtain essentially the same theoretical complexity bounds. More precisely, we can use step-sizes
\[
    \eta_t = \frac{\sqrt{2}\left( \sup_{r,r'}\frac{B_{\mathcal{X}}(r,r')}{2\log n} + \sup_{f,f'}\frac{B_{Y}(f,f')}{2\overline{R}^2} \right)^{1/2}}{L\sqrt{t}} = \frac{\sqrt{3}}{L\sqrt{t}}, t \geq 0
\]
and use the sliding average \(\widetilde r = (\sum_{t=1}^N \eta_t r^{(t)})/(\sum_{t=1}^N \eta_t)\) as an output of our algorithm to obtain a fully online algorithm for an infinite horizon.

\subsection{Linear Kernel Mirror Descent}
\label{sec:linear_kmd}
In this subsection, we study a simple kernel that allows easy computation of \(f_\mu\) and all related constants. The price for this simplicity is a poor family of optimized functions: it is a linear kernel \(\mathcal{K}(x,x') = \langle x, x' \rangle\) and a family of linear operators \(\Delta^n \mapsto \mathbb{R}^n\). For this kernel we can compute both constants:
\[
    \kappa^2 = \sup_{x \in \Delta^n} \langle x,x \rangle = 1, \ \ \ R^2 = \sup_{f_\mu \in \mathcal{H}^n_b} \Vert f_\mu \Vert_{2}^2 \leq 2 n^2 \Vert C \Vert_\infty^2,
\]
where we used representation of a linear operator as a matrix of size \(n \times n\). Overall, in this setting we have the following complexity in terms of the number of gradient computations:
\[
    N = O\left( \frac{n^4 \Vert C \Vert^2_\infty}{\varepsilon^2} \log^2\left(\frac{1}{\sigma}\right)\right).
\]
Moreover, each computation of the gradient can be done on \(O(n^2)\) time since any function \(f_\mu\) can be represented by a square matrix \(\Theta\), and each gradient step can be represented in the space of matrices.
This leads to the following total complexity
\[
  N = O\left( \frac{n^6 \Vert C \Vert^2_\infty}{\varepsilon^2} \log^2\left(\frac{1}{\sigma}\right)\right).
\]
On the one hand, this theoretical complexity is quite pessimistic. On the other hand, the main advantage of our algorithm with the linear kernel is the lack of parameters to tune, and relatively cheap iterations in comparison to KMD with general kernels and methods based on gradients of the optimal transport distance \cite{dvinskikh2020stochastic}.

\begin{remark}
    All the complexity bounds in sections~\ref{sec:fin_support},~\ref{sec:general_kmd},~\ref{sec:linear_kmd} are based on bounds for the norm of the solution to the dual problem \eqref{eq:dual_problem} and for the dual norm of stochastic subgradients. These bounds are typically rough enough and we may expect that in practice significantly smaller estimates can be used, which may lead to faster convergence  due to larger step-sizes. From this point of view, it may be useful to use adaptive variants of Stochastic Mirror Descent or Stochastic Mirror Prox): \cite{stonyakin2018generalized,bach2019universal,stonyakin2021inexact} with the adaptive batching idea of \cite{dvinskikh2019adaptive,dvinskikh2020line-search}.
\end{remark}

\section{Experiments}
\label{sec:experiments}

{
In this section, we illustrate the practical performance of the proposed algorithms by numerical experiments. Our goal is to compute population barycenter with respect to \(\W_2\) distance on various datasets.
}

\subsection{Finite Support Algorithm}
\label{sec:finite_support_exp}

{
Algorithm \ref{alg:finite_support} for the finite-support setting has good theoretical properties but its practical performance depends on the choice of the step-size. To demonstrate this effect, we perform experiments on an auxiliary dataset of \(m = 10\) 1-d Gaussian measures. Each measure is discretized over the segment \([-10; 10]\) with \(n = 300\) point masses. We use a small number of measures to make computation of the objective functional feasible. 

We vary dependence on the total number of steps \(N\) in the step-size formula of Algorithm \ref{alg:finite_support} from \(N^{-0.5}\) to \(N^{-0.33}, N^{-0.2}, N^{-0.1}\) and perform computation with \(N = 5,000,000\) steps. We demonstrate the rate of convergence for all these step-size choices in Figure \ref{fig:fin_sup_conv}. Additionally, we compare shapes of the final measures in the case of a step-size \(\sim N^{-0.33}\) with the barycenter given by the solution of the corresponding linear program and with output of the Iterative Bregman Projections (IBP) algorithm with \(\gamma = 3 \cdot 10^{-4}\) and \(\gamma = 10^{-2}\) in Figure \ref{fig:fin_sup_bary}. 

In the case of larger step-size (\(\sim N^{-0,2}\) and \(\sim N^{-0.1}\)) we observe problems with scaling: there is some point with much bigger assigned mass than other. 
At the same time for IBP, smaller values of the regularization parameter lead to numerical instability of the algorithm whereas the value \(\gamma = 10^{-3}\) gives us a very similar picture to \(\gamma = 3 \cdot 10^{-4}\).  
We also can see that IBP outperforms our approach in terms of convergence rate.

\begin{figure}
    \centering
    \includegraphics[scale=0.34]{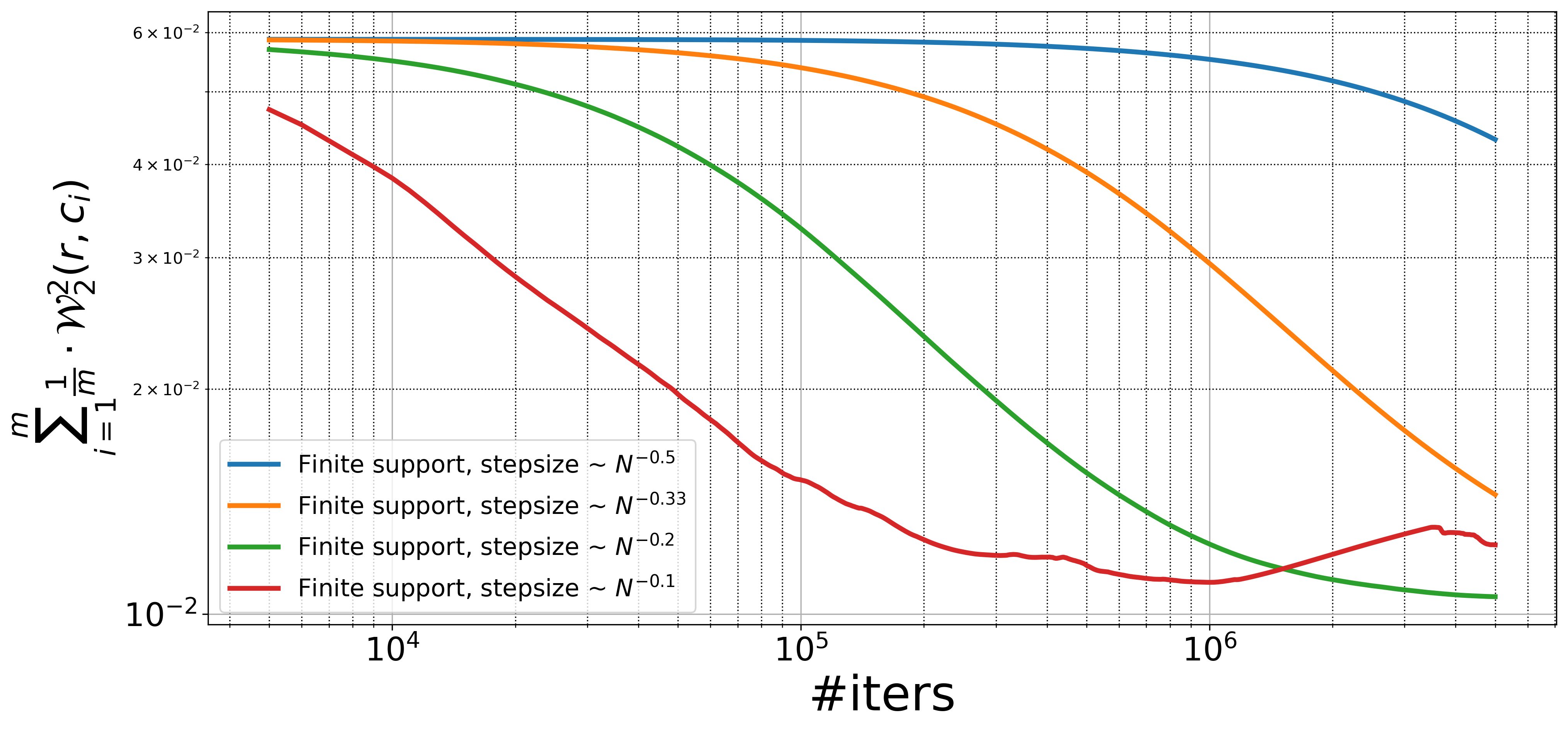}
    \caption{\(\W_2\)-barycenter functional value with various step-size parameters in Algorithm \ref{alg:finite_support}.  Presented in log-log scale.}
    \label{fig:fin_sup_conv}
\end{figure}

\begin{figure}
    \centering
    \includegraphics[scale=0.35]{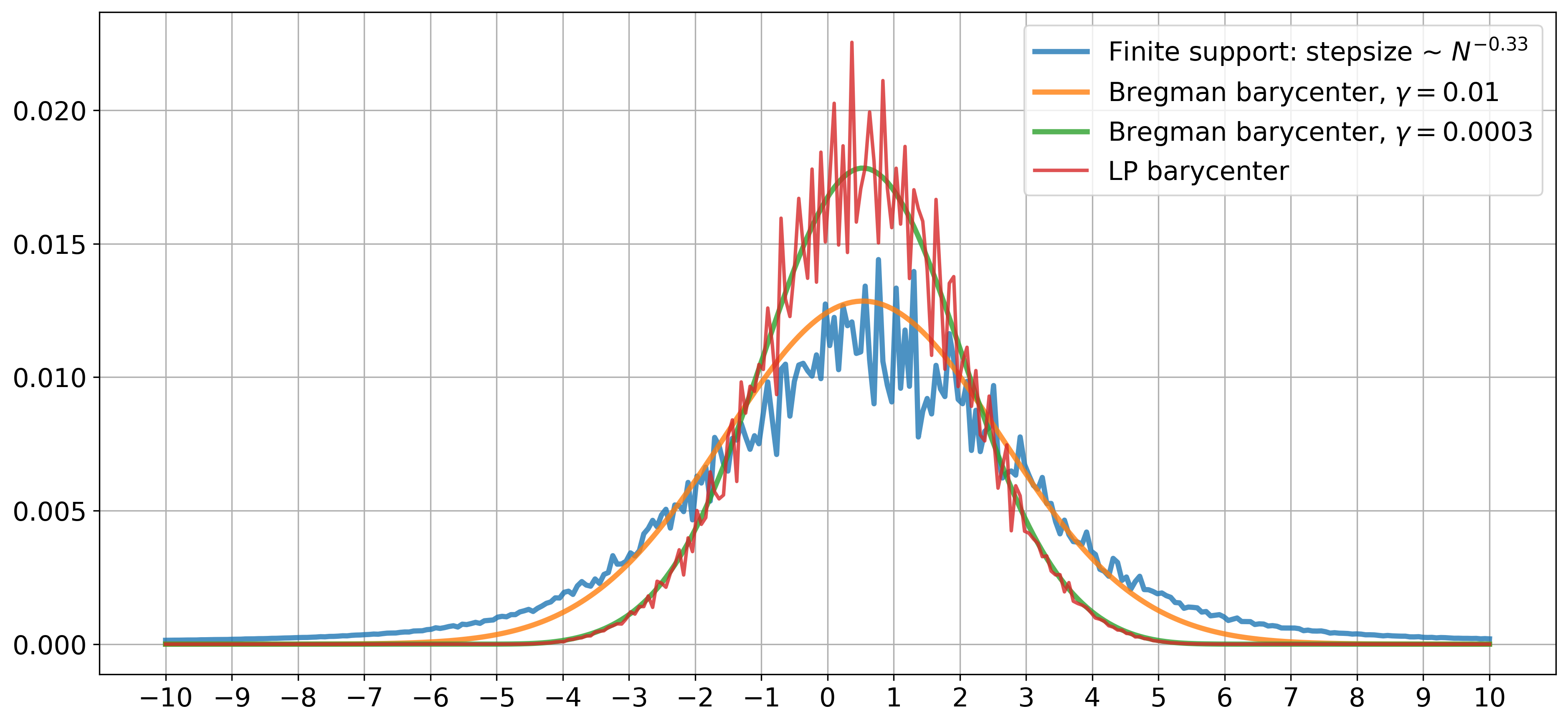}
    \caption{Visual comparison between Algorithm \ref{alg:finite_support}, IBP and solution of LP.}
    \label{fig:fin_sup_bary}
\end{figure}
}

\subsection{Kernel Mirror Descent.}
\label{sec:kmd_exp}

We study Kernel Mirror Descent Algorithm \ref{alg:general_kmd} in three settings:
\begin{itemize}
    \item Gaussian RBF kernel \(\mathcal{K}(x,x') = \exp\left( -s \Vert x - x' \Vert_2^2\right)\);
    \item Information-Diffusion kernel \(\mathcal{K}(x,x') = \exp\left( -\frac{1}{t} \arccos^2\left( \langle \sqrt{x}, \sqrt{x'} \rangle\right) \right)\);
    \item Linear kernel \(\mathcal{K}(x,x') = \langle x, x' \rangle\).
\end{itemize}

For the first two settings we choose a good value of the parameter \({\overline{R}}^2\) and the parameter of the kernel by performing a grid-search. 
For the last setting, we implement a more computationally efficient algorithm that works directly with matrices as briefly discussed in Section \ref{sec:linear_kmd}. 

We compare our approach with two algorithms, described in \cite{dvinskikh2020stochastic}:
\begin{itemize}
    \item Gradient descent with gradients of the regularized Wasserstein distance computed by \new{Sinkhorn algorithm with stabilization \cite{schmitzer2019stabilized}};
    \item Mirror descent with subgradients of non-regularized Wasserstein distance  computed by solving the dual OT problem via linear programming solver.
\end{itemize}
We refer to the first approach as Sinkhorn-based and the second one as direct LP-based.

\paragraph{1-d Gaussian measures.}
To compare convergence we consider an auxiliary dataset of randomly generated Gaussian measures. The main reason to choose this dataset is the ability to compute the true {\(\W_2\)-}barycenter of the distribution \cite{delon2020wassersteintype}. We generate 10000 1-d Gaussian measures \(\mathcal{N}(\mu,\sigma^2)\) with \(\mu\) sampled from \(\mathcal{N}(1,4)\) and \(\sigma\) sampled from \(\operatorname{Exp}(1/2)\). The true population barycenter for this distribution over Gaussian measures is \(\mathcal{N}(1,4)\), since the expected value of \(\mu\) is 1 and the expected value of \(\sigma\) is 2.

After that, for each measure, we consider its discretization over the segment \([-10;10]\) with \(n=300\) point masses. In the first step, we consider distributions over mean and variance with much more general support to generate measures that could be called outliers, and make this model closer to applications. {Our goal is to compute approximate \(\W_2\)-barycenter using these discrete measures.}

As a performance measure we choose \(\W_2\)-distance to the true barycenter. We notice that we do not have convergence guarantees in this metric for our algorithms. However, this is one of the very limited number of ways to evaluate the quality of a barycenter approximation in the online approach where the expectation can not be expressed analytically. 

{
Firstly, we found optimal parameters for our general kernel methods by the 2-dimensional search over log-scale grid. For both considered types of general kernels the optimal value was \(\overline{R}^2 = 45\), optimal \(s\) for RBF kernel is \(0.02\) and optimal \(t\) for information diffusion kernel is \(200\). We observe a tradeoff between the stability of the learning trajectory and the functional optimization, as can be seen in Figure \ref{fig:gridsearch}. The key effect is that it is sufficient to choose a relatively small value of \(\overline{R}^2\) to have good convergence properties. \new{Notice that we do not claim that our algorithm has convergence guarantees in \(\W_2\)-distance and the effect of stagnation on the level about \(5 \cdot 10^{-2}\) in Figure \ref{fig:gridsearch} could be explained by a lack of convergence in \(\W_2\)-distance.}

\begin{figure}
    \centering
    \includegraphics[scale=0.35]{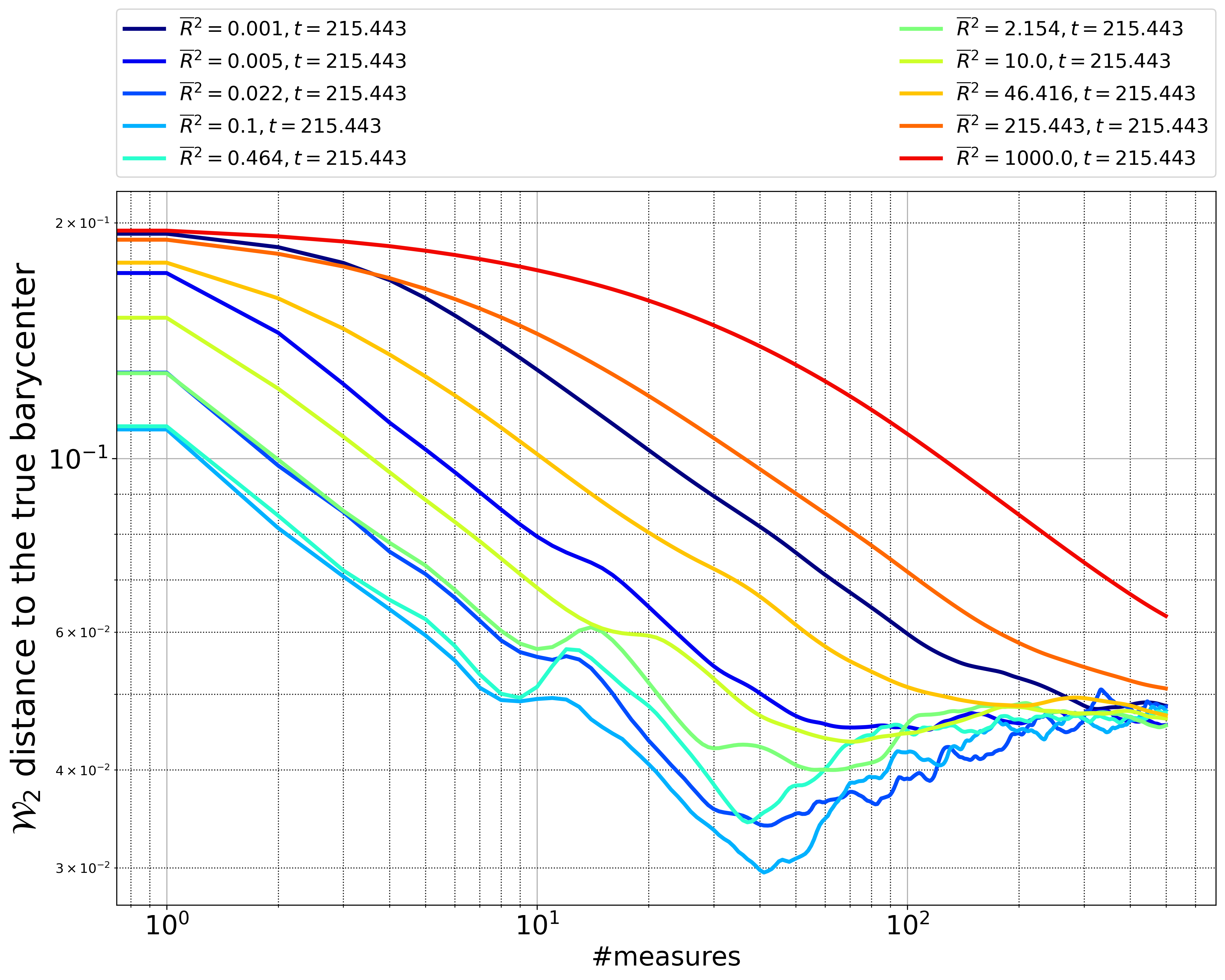}
    \caption{Change of parameter \(\overline{R}^2\) with fixed parameter of information-diffusion kernel. Presented in log-log scale.}
    \label{fig:gridsearch}
\end{figure}

Secondly, we compare the barycenter approximations obtained by our two kernel methods with the Sinkhorn-based approach with three values of the regularization parameter \(\gamma = 10^{-2}, 10^{-4}, 10^{-6}\). Since the barycenters produced by Kernel Mirror Descent with both studied kernels are visually almost the same, we present only one of them. We also perform the comparison for \(\gamma = 10^{-3}\) and \(\gamma = 10^{-5}\), but we choose previous three values as the most illustrative ones. The result of comparison can be found in Figure \ref{fig:kmd_vs_sinkhorn}.

\begin{figure}
    \centering
    \includegraphics[scale=0.35]{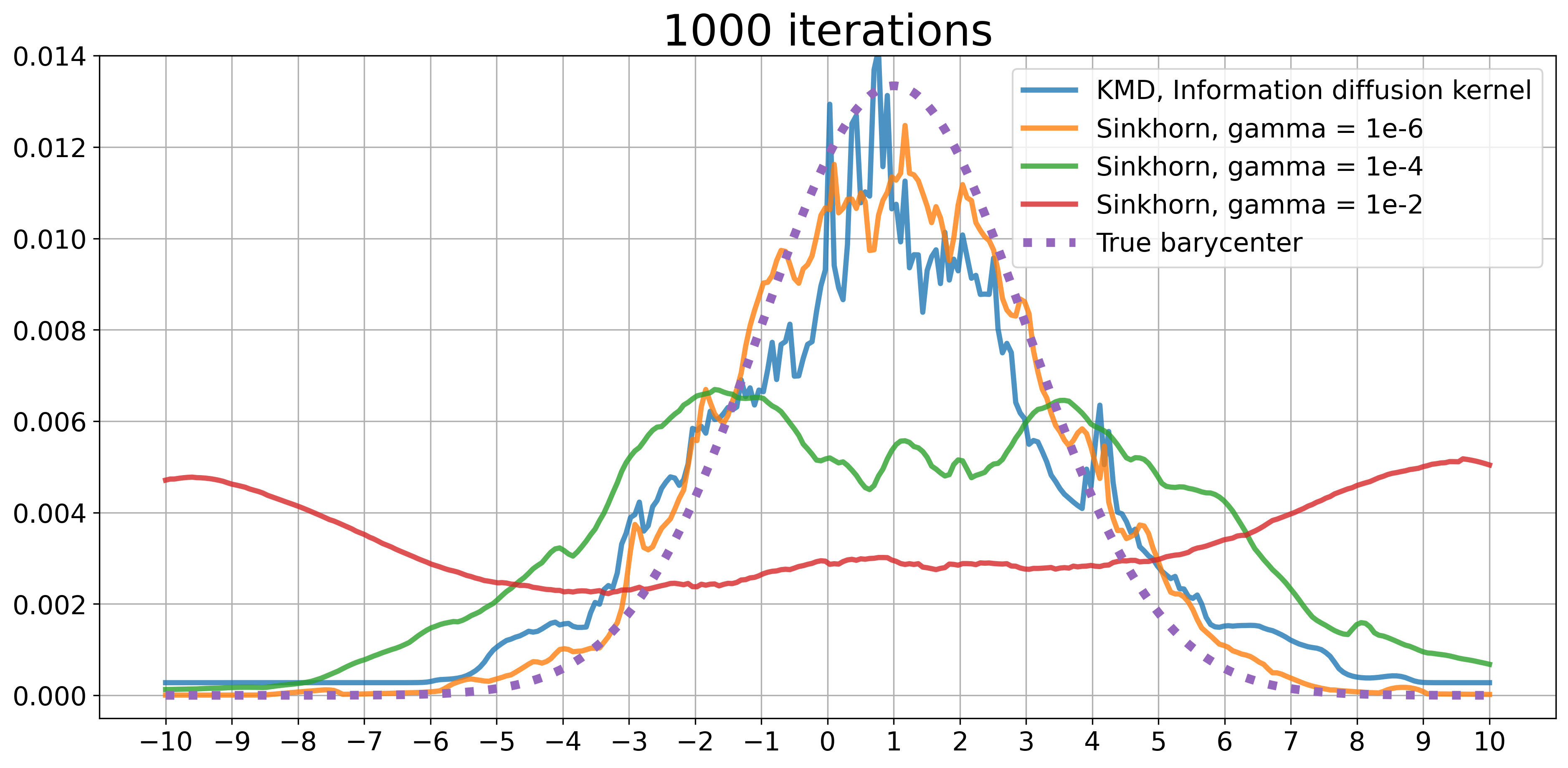}
    \includegraphics[scale=0.35]{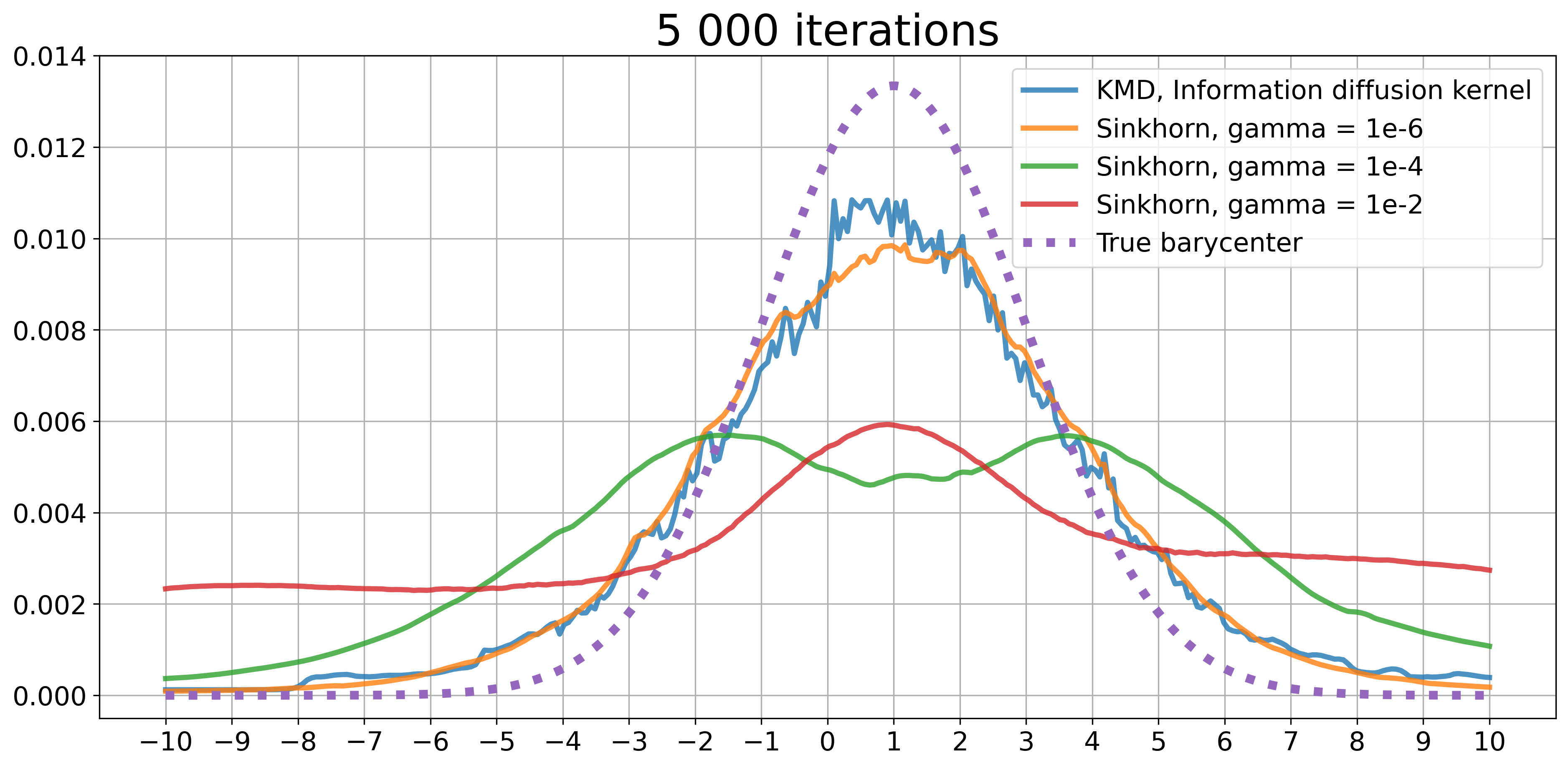}
    \includegraphics[scale=0.35]{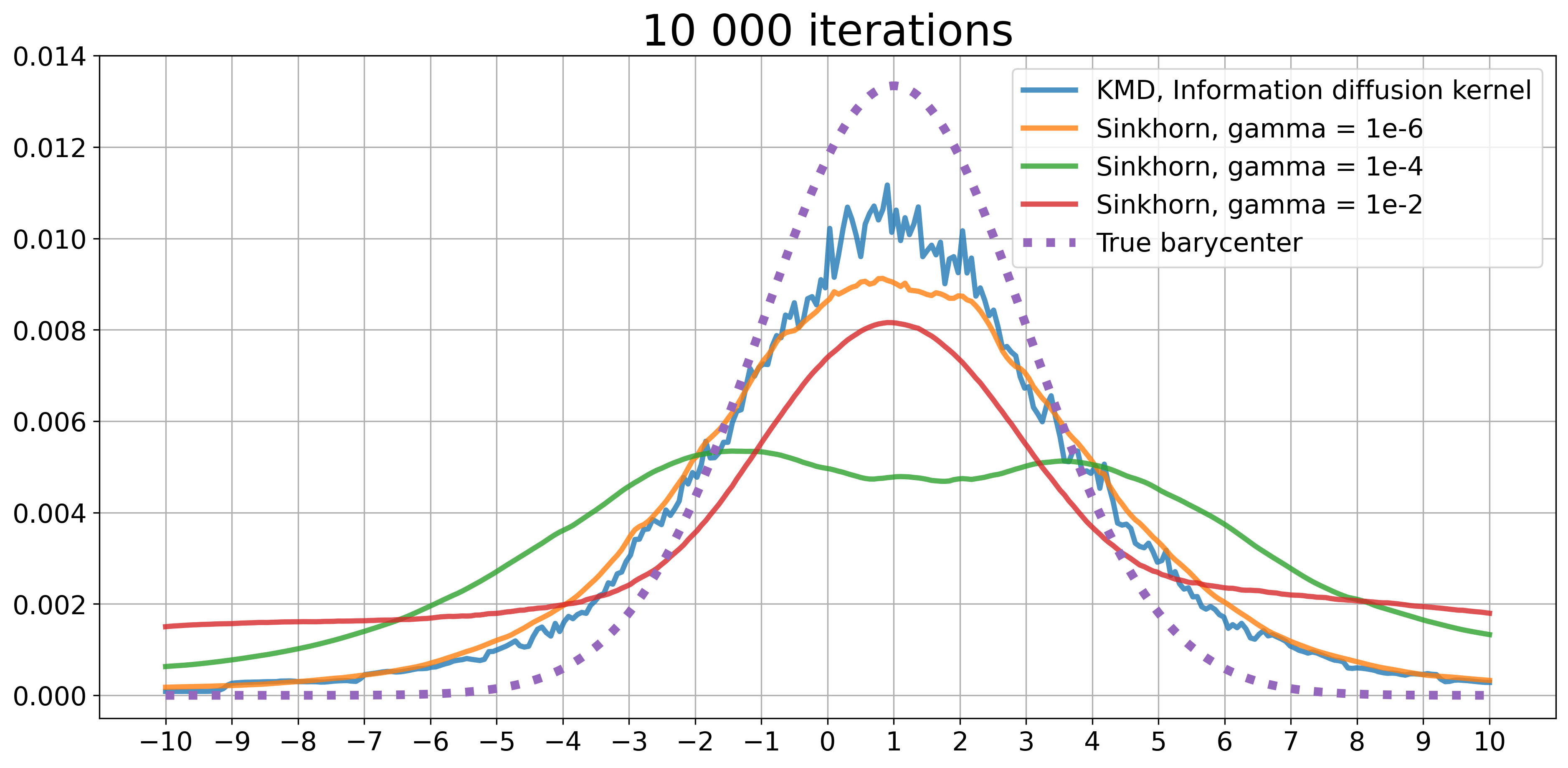}
    \caption{Comparison between the KMD and the Sinkhorn-based method.}
    \label{fig:kmd_vs_sinkhorn}
\end{figure}

For Sinkhorn-based methods, we faced numerical instabilities in the computation of the gradients that may lead to numerical errors \new{even using a stabilized algorithm}.  
This computational issue may explain the strange shape of the approximate barycenter produced by Sinkhorn-based method with \(\gamma = 10^{-4}\).
Also, notice that the gradient descent with Sinkhorn-calculated gradients converges not to the true but to the regularized barycenter. 

Next, we compare barycenter approximations computed with the linear KMD and the Sinkhorn-based method in Figure \ref{fig:lin_kmd_vs_sinkhorn}. Note that our algorithm outperforms Sinkhorn-based method with a relatively big value of \(\gamma = 0.01\). For the very small value \(\gamma = 10^{-6}\) the results are close with our approximation being a bit sharper. Moreover, our algorithm is parameter-free and we do not require to tune the regularization parameter \(\gamma\) and the number of Sinkhorn iterations to have a relatively good approximation quality.

\begin{figure}
    \centering
    \includegraphics[scale=0.35]{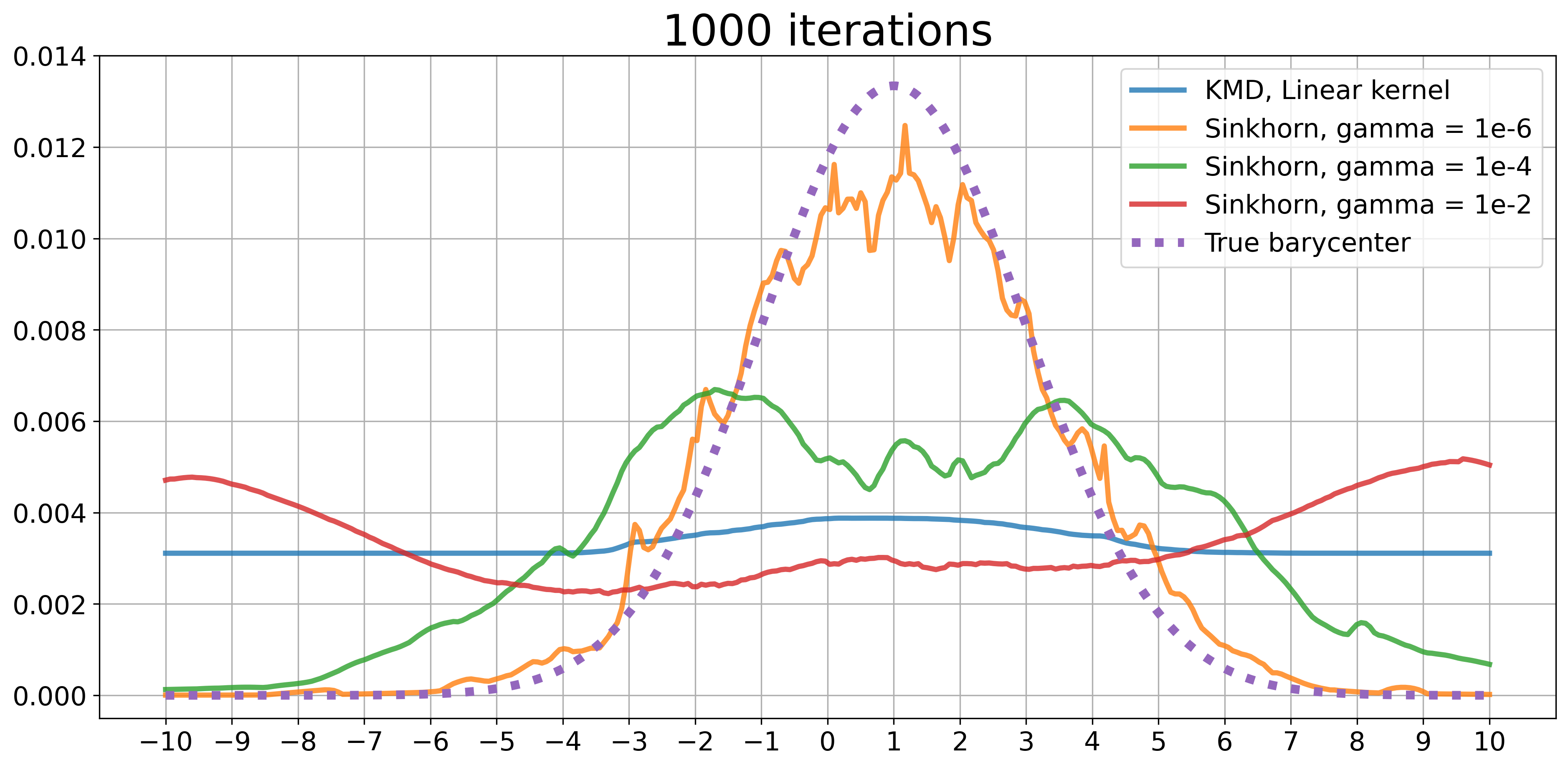}
    \includegraphics[scale=0.35]{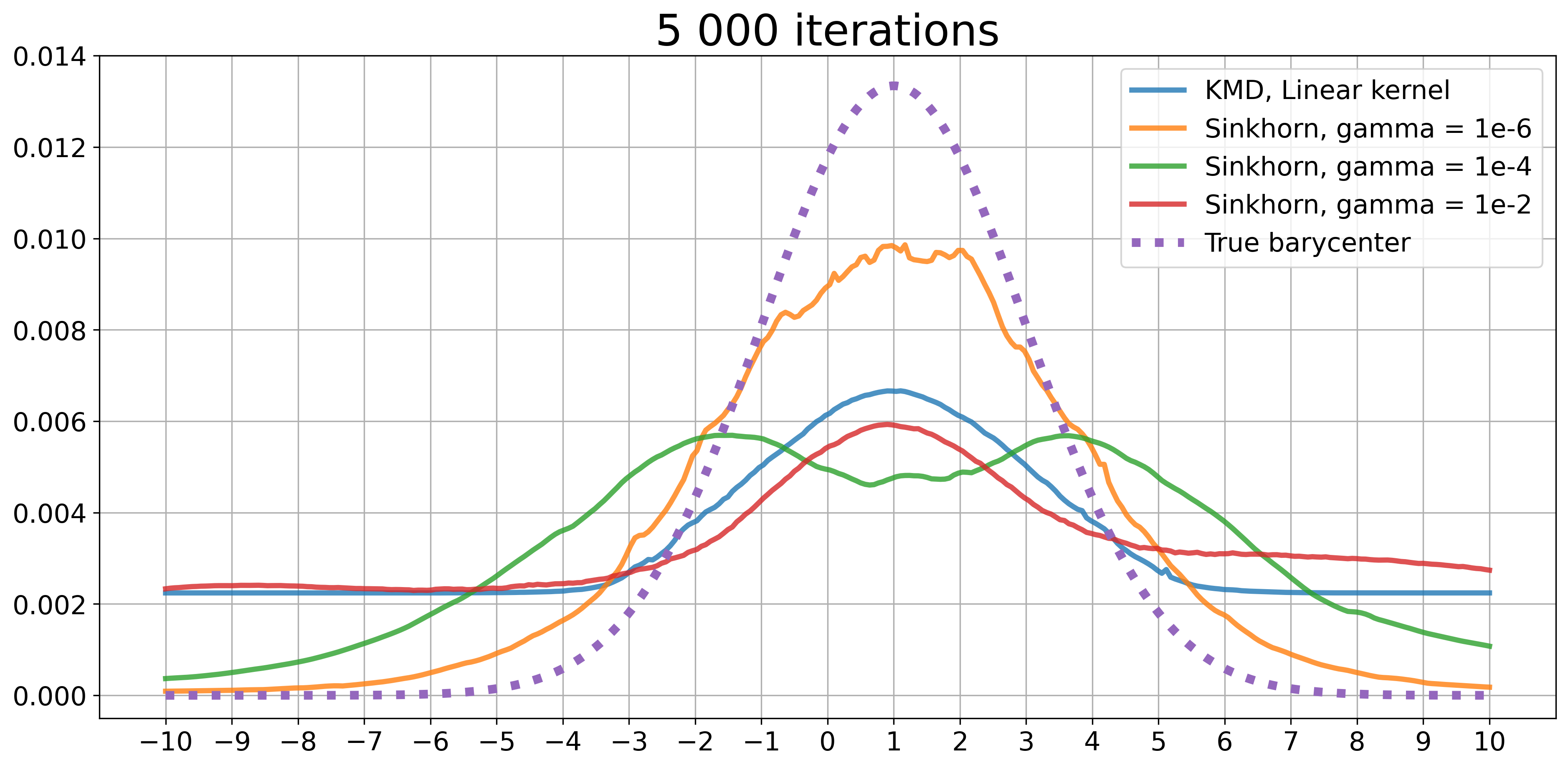}
    \includegraphics[scale=0.35]{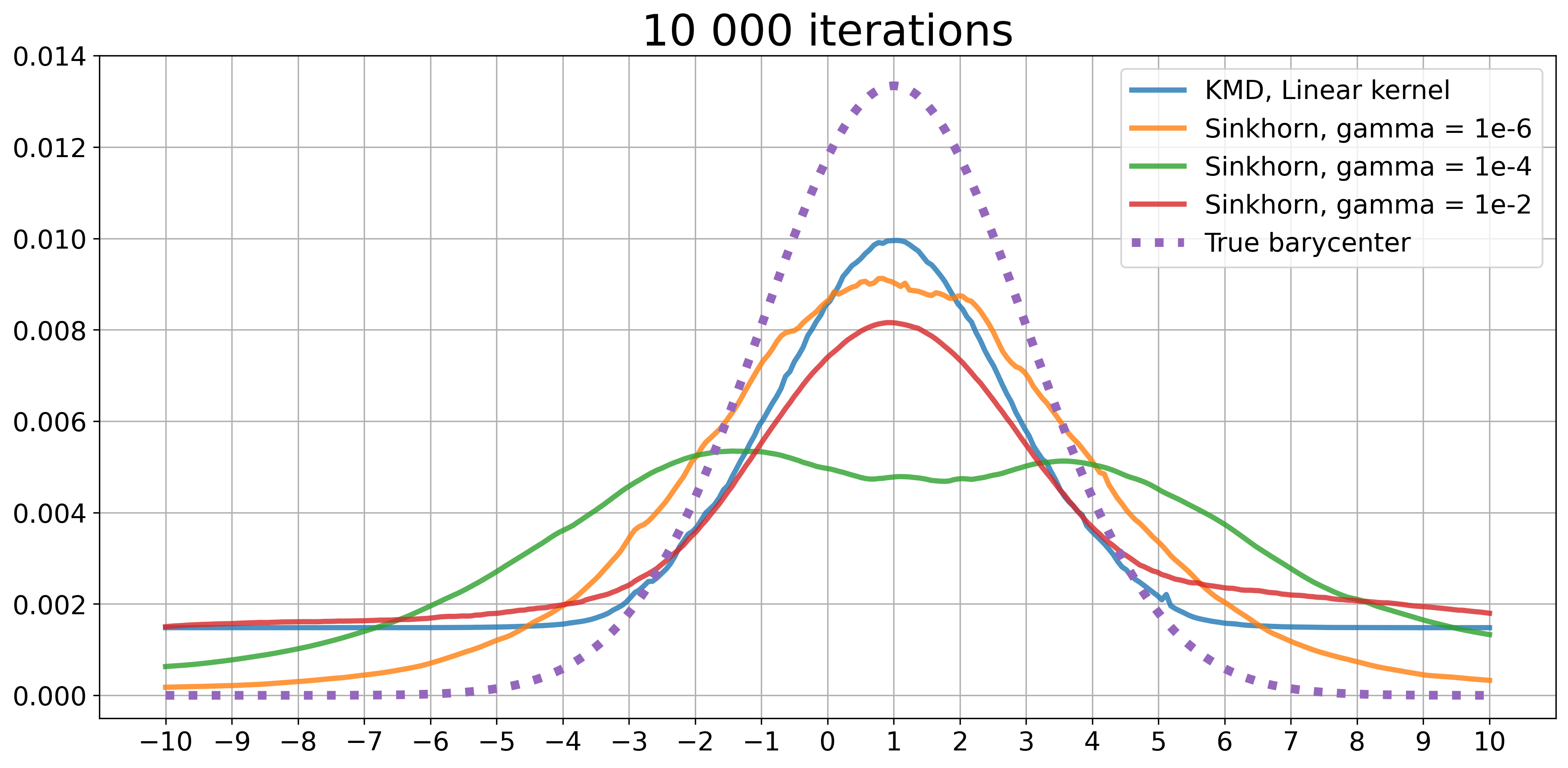}
    \caption{Comparison between the linear KMD and the Sinkhorn-based method.}
    \label{fig:lin_kmd_vs_sinkhorn}
\end{figure}

A comparison of the approximation by our method with the direct LP-based method is presented in Figure \ref{fig:kmd_vs_lp}, where we can see that \new{to obtain a good quality, the direct LP-based method requires more measures than KMD.} 

\begin{figure}
    \centering
    \includegraphics[scale=0.35]{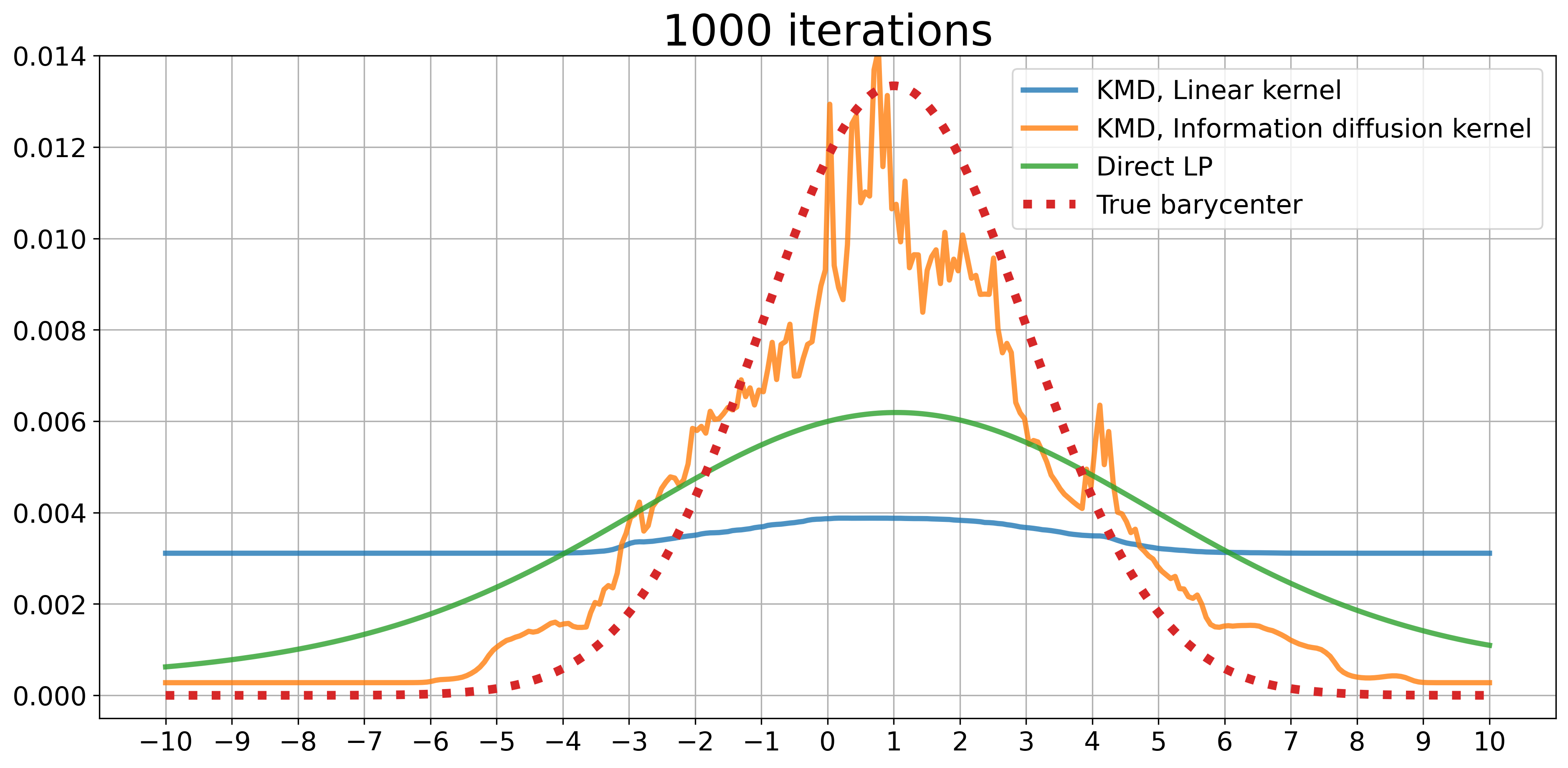}
    \includegraphics[scale=0.35]{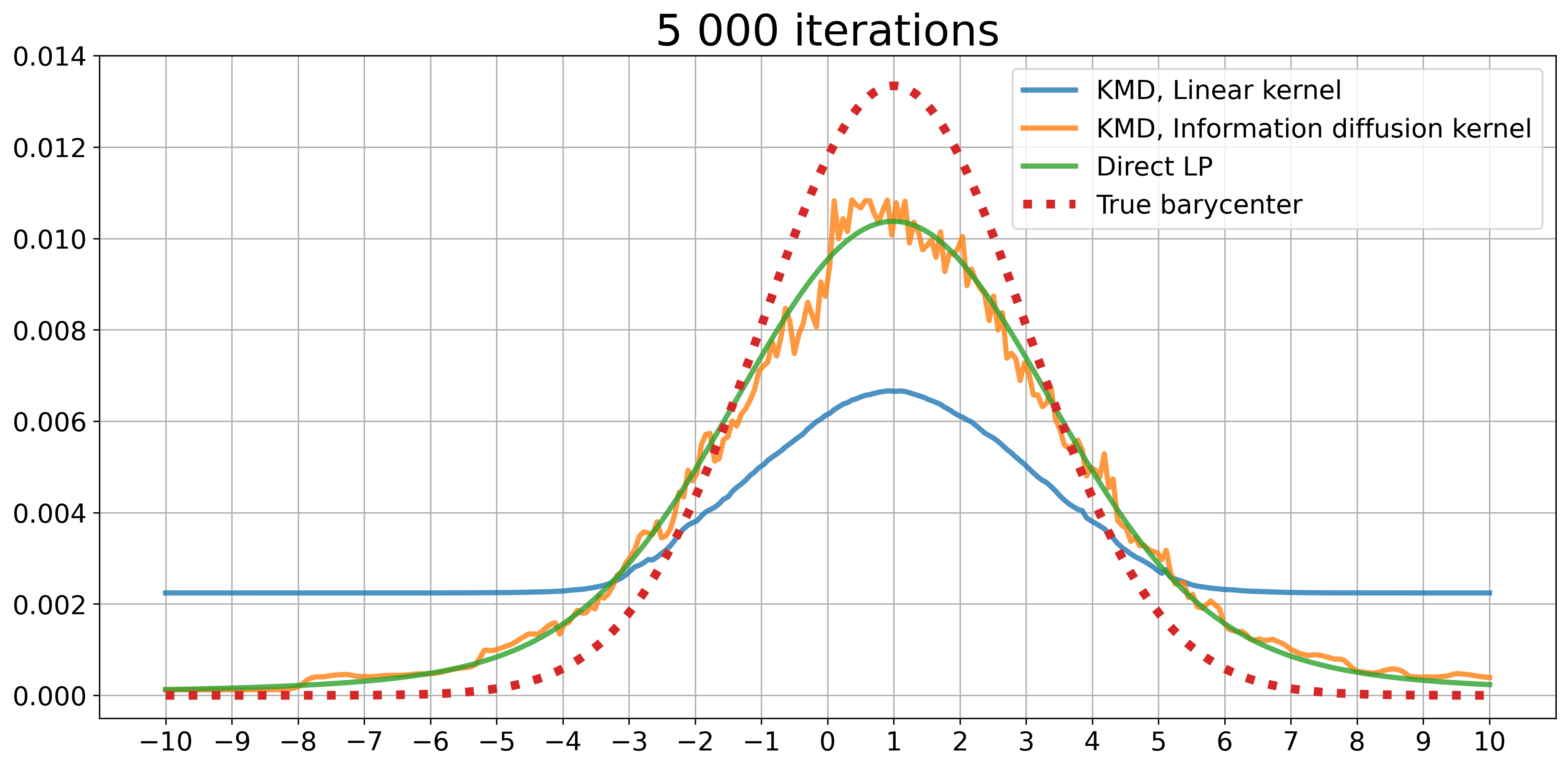}
    \includegraphics[scale=0.35]{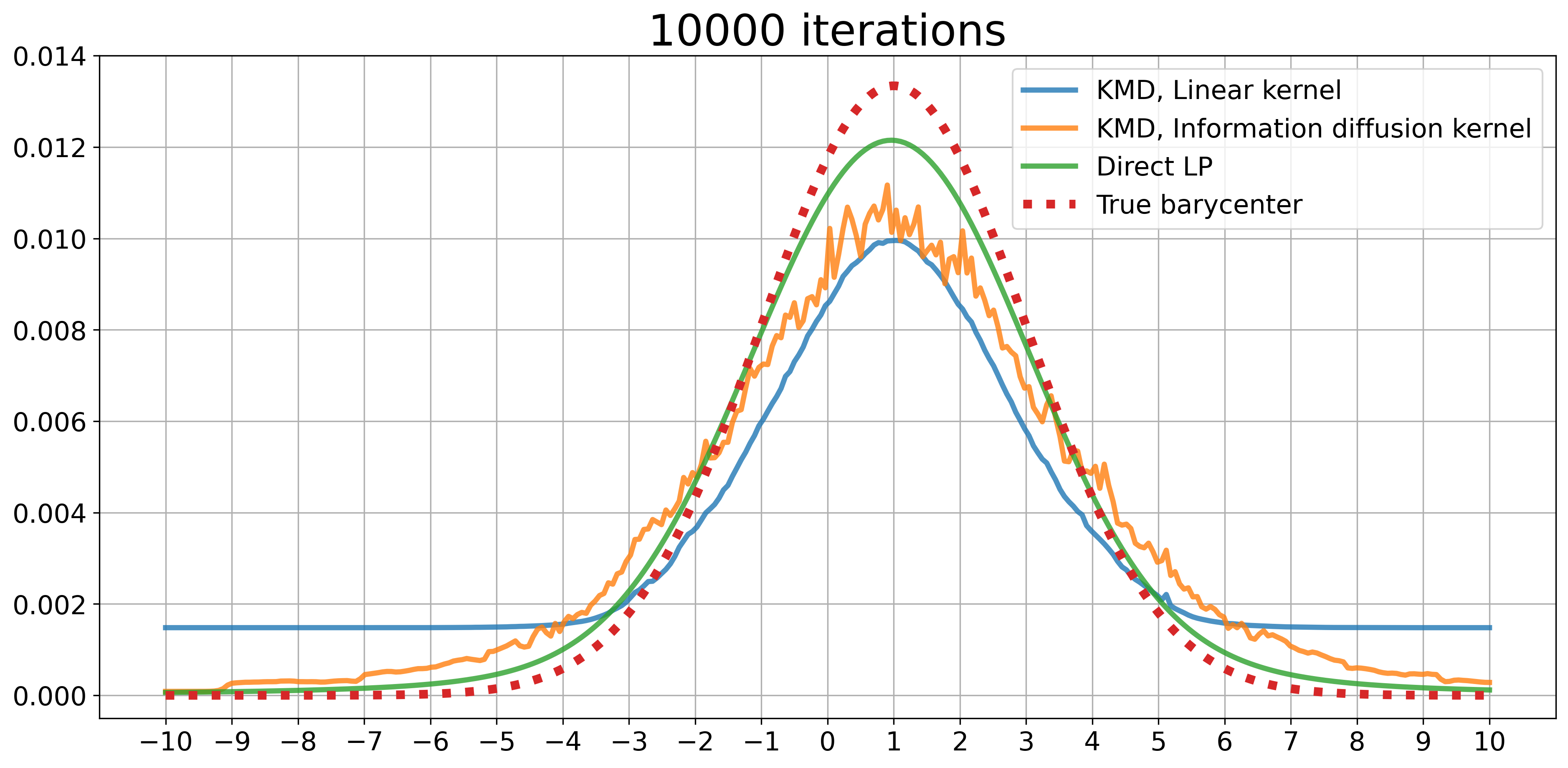}
    \caption{Comparison between the KMD and the direct LP-based method.}
    \label{fig:kmd_vs_lp}
\end{figure}

Finally, we compare all the methods using the proposed quality metric in Figure \ref{fig:w2_dist_measures}. Gaussian and information diffusion kernels show almost the same performance. The best method in \(\W_2\) distance to the true barycenter is the direct LP-based method. 

Also, we can observe that the Sinkhorn-based method with small \(\gamma =  10^{-6}\) has an increasing trend in terms of \(\W_2\)-distance to the true barycenter. This may be due to the convergence to the regularized barycenter and due to numerical instabilities in gradients caused by the small value of \(\gamma\): the computation of the gradient becomes less and less precise.

Additionally, we can see that approximate barycenters produced by the linear KMD outperforms Sinkhorn barycenters with \(\gamma = 10^{-2}\) and \(\gamma = 10^{-4}\) not only visually but also in the \(\W_2\) distance, and the barycenter approximations produced by the KMD with Gaussian kernel become better than the Sinkhorn-based method when the iteration number is large.

\begin{figure}
    \centering
    \includegraphics[scale=0.35]{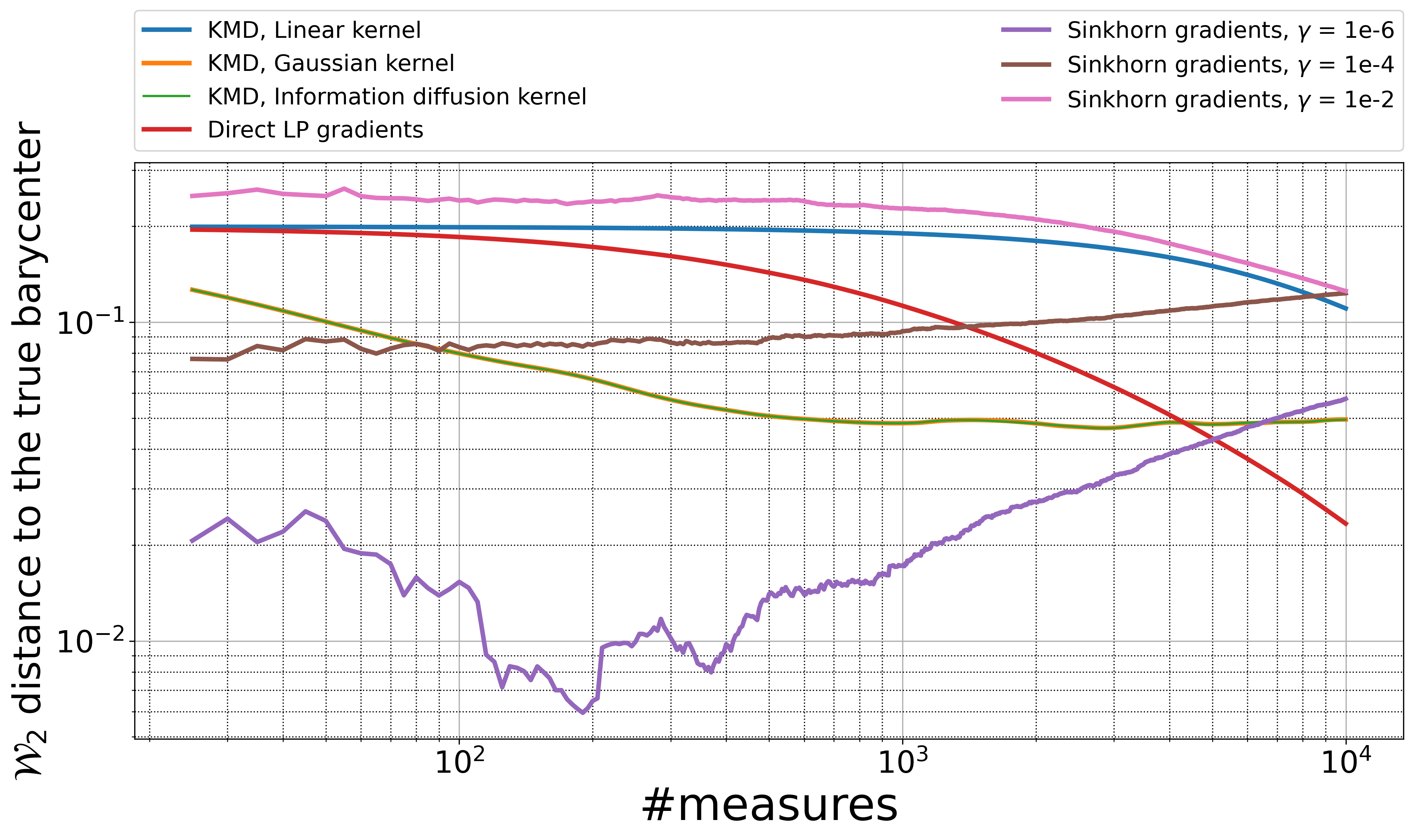}
    \caption{Quality of the estimation in \(\W_2\)-distance to the true barycenter on 1-d Gaussian dataset. Presented in log-log scale.}
    \label{fig:w2_dist_measures}
\end{figure}

We compare also the running time of each algorithm in Figure \ref{fig:processing_of_measures} except the direct LP-based method \new{and Sinkhorn-based methods with \(\gamma = 10^{-4}\) and \(\gamma = 10^{-6}\)} since they spend about \(132, 24\) and \(25\) minutes respectively whereas all the other methods converge in less than 4 minutes. We observe one expected effect: the curves corresponding to the general KMD are parabolas whereas all other curves are lines. \new{Another effect on this graph is related to the regularization parameter \(\gamma\) for the Sinkhorn-based method: smaller parameter value requires more iterations to produce a gradient.} Also we can see that the information diffusion kernel is slightly more computationally efficient than the Gaussian kernel.
}
\begin{figure}
    \centering
    \includegraphics[scale=0.35]{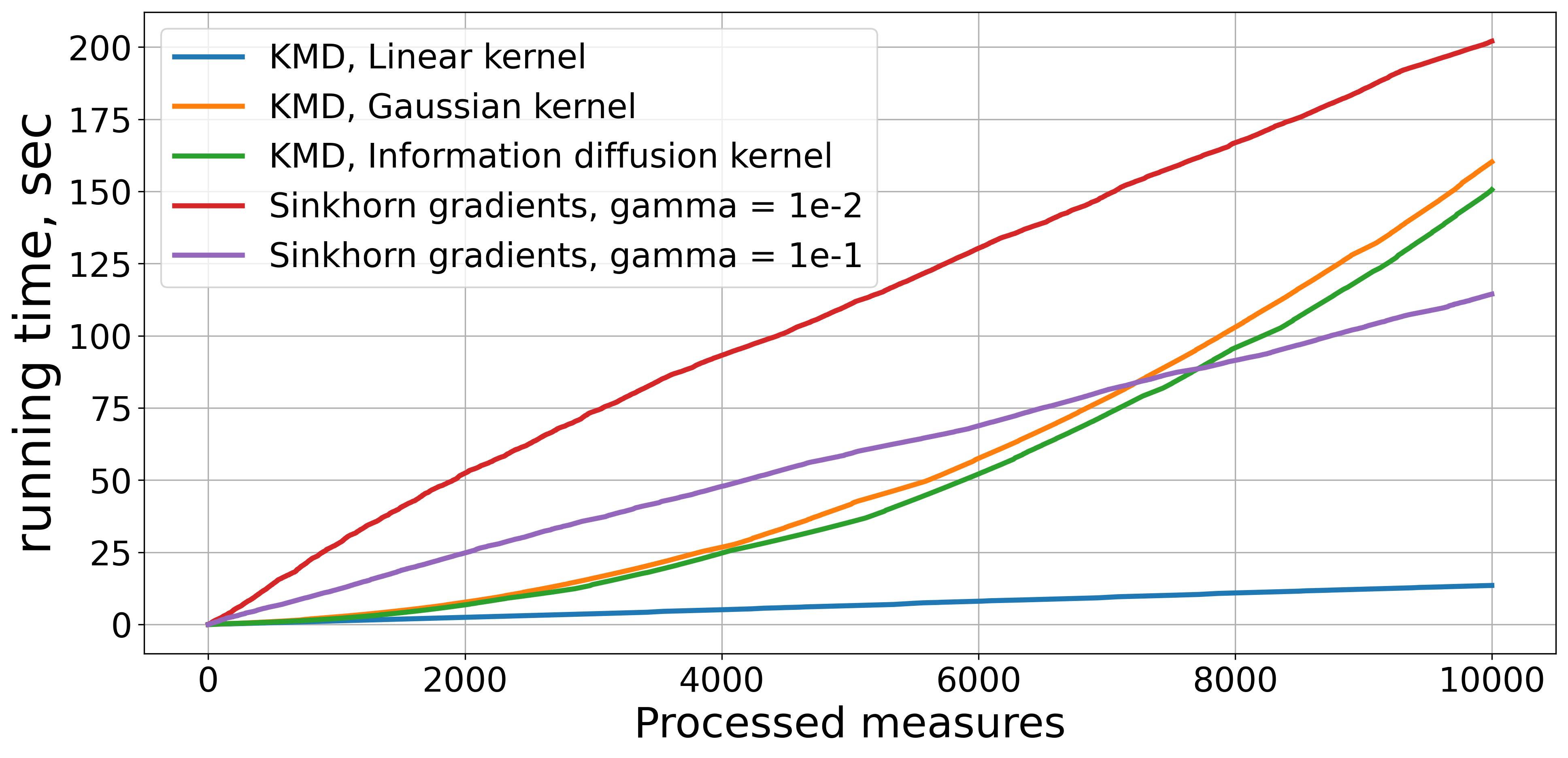}
    \caption{Speed of processing new measures on 1-d Gaussian dataset.}
    \label{fig:processing_of_measures}
\end{figure}

\paragraph{2-d Gaussian measures.}
Our next goal is to evaluate performance of our algorithms in the high-dimensional setting. Our choice is a dataset of randomly generated  2-dimensional Gaussian measures discretized over the grid \(50 \times 50\). Thus, our input histogram measures are elements of a 2500-dimensional simplex.

In this case, LP-based approach is computationally infeasible since the complexity of gradient computation by an LP solver is \(O(n^3)\). Thus, we compare the proposed approach with the Sinkhorn-based online method.

First of all, we describe the measure generation process: we sample means of 2-dimensional Gaussians from \(\mathcal{N}((-1,1)^\top, 0.75)\) and covariance matrices from Inverse-Wishart distribution with parameters \(\Phi = \begin{pmatrix}
    1 & 0.1 \\
    0.1 & 0.3
\end{pmatrix}\) and \(\nu = 2\). To compute the barycenter of these distributions, we sample \(2 \cdot 10^6\) measures and apply the fixed-point iteration procedure \cite{delon2020wassersteintype}. After that, we compute histograms of all these measures at the segment \([-5, 5]^2\) over grid \(50 \times 50\), add 1e-9 to all elements and normalize to ensure positivity. 

To compare our algorithms with the Sinkhorn-based method we used \(5000\) measures. The comparison in \(\W_2\)-distance to the true barycenter can be found in Figure \ref{fig:2d_quality}. The best result is obtained by the Sinkhorn-based algorithm with \(\gamma = 10^{-6}\). Also, notice that our approaches with general kernels perform well whereas linear KMD almost failed. 

\new{Running time comparison is presented in Table \ref{tab:2d_speed}. Notice that values for the smallest regularization parameters almost coincide because algorithms reached a maximum number of Sinkhorn iterations. In this table, we see that KMD performs much faster than Sinkhorn-based method whereas it has a relatively good quality that was outperformed only by using very small \(\gamma = 10^{-6}\) with almost \(50\)-time worse running time.}

\begin{figure}
    \centering
    \includegraphics[scale=0.35]{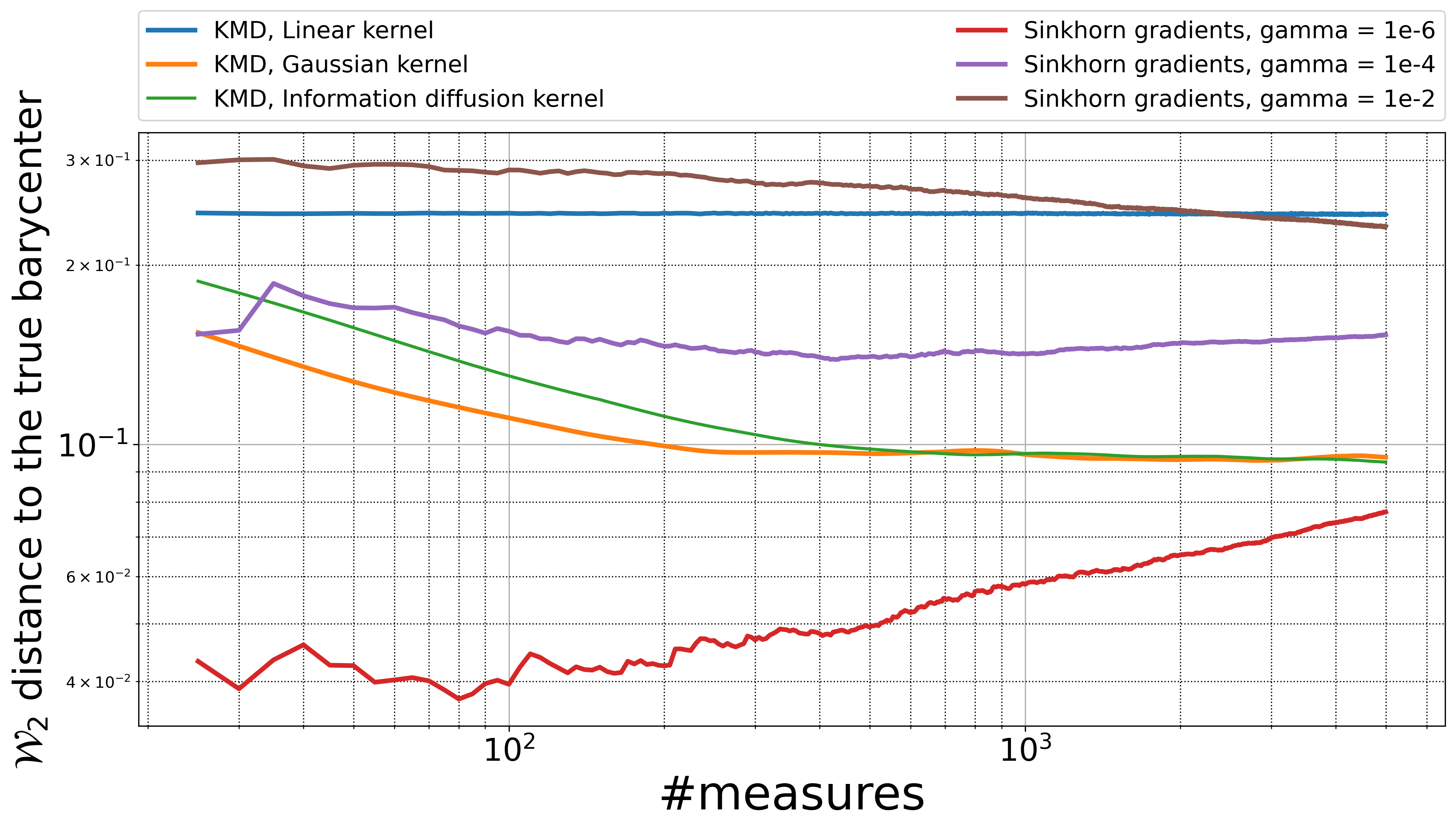}
    \caption{Quality of the estimation in \(\W_2\)-distance to the true barycenter on 2-d Gaussian dataset. Presented in log-log scale.}
    \label{fig:2d_quality}
\end{figure}

\begin{table}
    \centering
    \begin{tabular}{|c|c|c|c|c|c|c|c|c|c|}
        \hline 
        Sinkhorn, \(\gamma=\) & \( 10^{-6}\) & \(10^{-5}\) &  \(10^{-4}\) &  \(10^{-3}\) &  \(10^{-2}\) &  \(10^{-1}\)\\ 
        Time, min. & 521.85 & 491.88 & 502.63 & 433.81 & 354.13 & 30.01 \\ \hline\hline
        KMD & \multicolumn{2}{c|}{Information-diffusion kernel} & \multicolumn{2}{c|}{RBF kernel} & \multicolumn{2}{c|}{Linear kernel} \\
        Time, min. & \multicolumn{2}{c|}{9.42} & \multicolumn{2}{c|}{10.32} & \multicolumn{2}{c|}{5.37}  \\
        \hline
    \end{tabular}
    \caption{Running time comparison for 2-d Gaussian dataset.}
    \label{tab:2d_speed}
\end{table}

\paragraph{MNIST dataset}
{ 
First of all, a real-world dataset gives us an opportunity to demonstrate the performance on real images. We use 2500 images of hand-written digits <<3>> from the MNIST dataset\footnote{\url{http://yann.lecun.com/exdb/mnist/}} and our goal is to compute their \(\W_2\)-barycenter in the online setting. The barycenter approximations generated by the Kernel Mirror Descent (Algorithm \ref{alg:general_kmd}) with an information-diffusion kernel with parameters \(\overline{R}^2 = 10, t = 10\) during first 100 iterations can be seen in Figure \ref{fig:mnist_learning}. Notice that the Sinkhorn-based method does not provide any meaningful pictures during the first 100 iterations.

\begin{figure}
    \centering
    \includegraphics[scale=0.28]{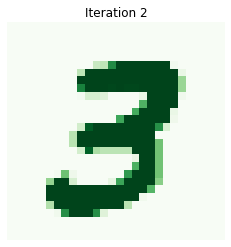}
    \includegraphics[scale=0.28]{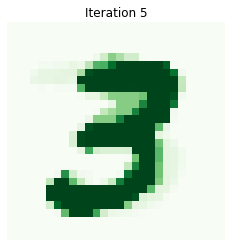}
    \includegraphics[scale=0.28]{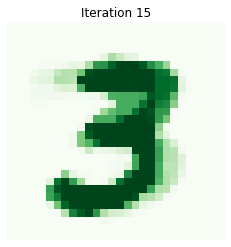}
    \includegraphics[scale=0.28]{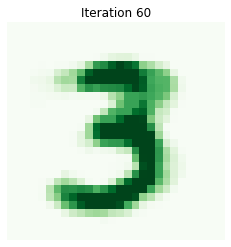}
    \includegraphics[scale=0.28]{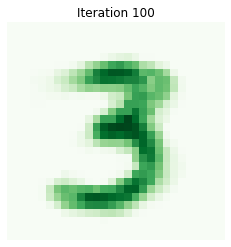}
    \caption{Learning of the barycenter of MNIST dataset.}
    \label{fig:mnist_learning}
\end{figure}

Our next goal is to study the influence of the kernel parameters and the parameter \(\overline{R}^2\) for the Kernel Mirror Descent algorithm with the information diffusion kernel. For this experiment, we run a 2-d grid-search over the parameters \(\overline{R}^2\) and \(t\). The most interesting results are presented in Figure \ref{fig:mnist_different_r_t}. 

For the values of \(t\) less than \(0.004\) the image disappears. Thus, we may say that the value of \(t \geq 1\) can be considered as a  <<safe>> threshold value for obtaining a good barycenter. However, there is a possibility to vary the <<thickness>> of the barycenter by tuning this parameter.

If we fix \(t\) and start to vary the parameter \(\overline{R}^2\), we can see that this parameter affects the <<sharpness>> of the barycenter. Additionally, the choice of  \(\overline{R}^2\) can be considered as a new type of regularization.

\begin{figure}
    \centering
    \includegraphics[scale=0.35]{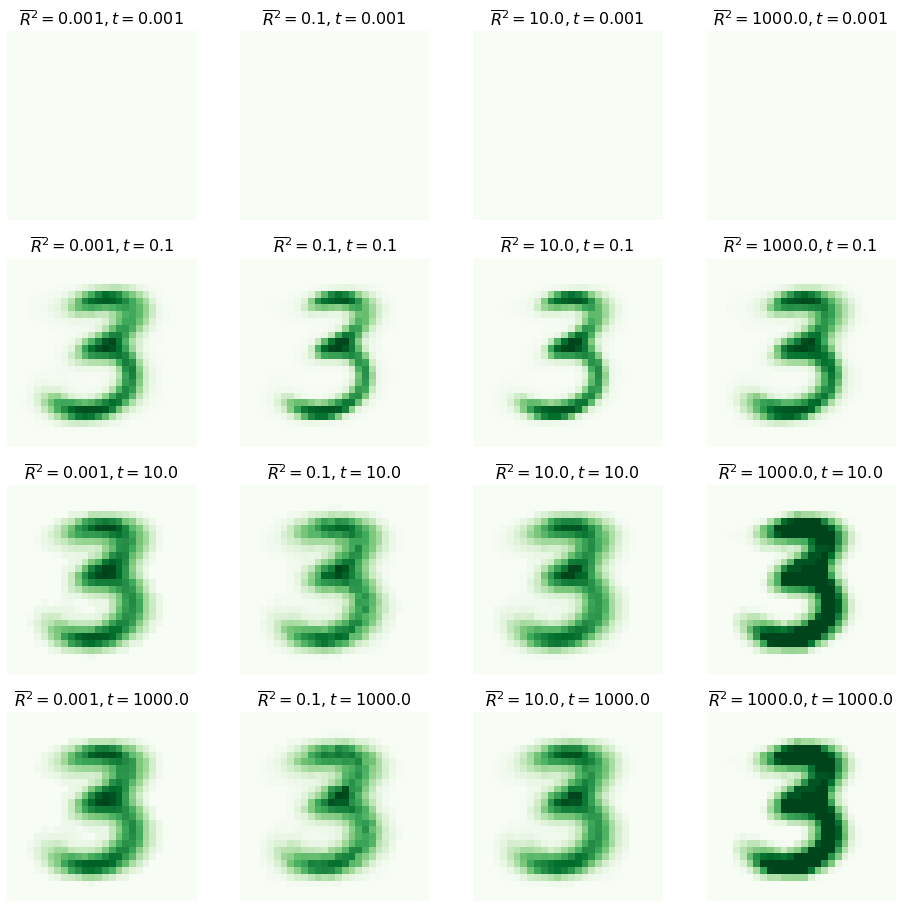}
    \caption{Comparison of the barycenter computed by the KMD with an information-diffusion kernel with different values of \(\overline{R}\) and \(t\).}
    \label{fig:mnist_different_r_t}
\end{figure}

}
\section{Conclusion}
\label{sec:conclusion}

In this work, we consider two algorithms for the stochastic approximation approach to the population Wasserstein barycenter problem in the discrete setting. This barycenter minimizes the expectation of the Wasserstein distance to random probability measure generated from some distribution.

The first algorithm uses an additional assumption that the expectation is taken over a finite number of measures. This assumption is very similar to the standard offline setting with the SAA-approach when we replace the expectation with the finite-sum. This algorithm has similar theoretical complexity to that of the well-known Iterative Bregman Projections algorithm. 

The second algorithm is proposed for a general setting without such restrictive assumptions and uses optimization in reproducing kernel Hilbert spaces as the main instrument. It is the first algorithm for the online Wasserstein barycenter problem that does not use typically expensive or numerically unstable computation of (sub)gradients of the Wasserstein distance or its regularized version. Moreover, in the experiment section, we show that our approach is practical and comparable to both main existing online approaches. Also, the use of the simplest possible linear kernel gives us a practical parameter-free algorithm.

Turning to possible extensions, a promising direction for future work is an application of this saddle-point framework to the entropy-regularized problem. Regularization allows one to use accelerated gradient methods that are known to be efficient for solving the regularized optimal transport problem \cite{kroshnin2019complexity,lin2019efficient}. Another reasonable direction is extension of this approach to the continuous measure spaces, e.g. the space of Gaussian measures.

\bibliographystyle{spmpsci}      
\bibliography{references.bib}   

\begin{thebibliography}{10}
\providecommand{\url}[1]{{#1}}
\providecommand{\urlprefix}{URL }
\expandafter\ifx\csname urlstyle\endcsname\relax
  \providecommand{\doi}[1]{DOI~\discretionary{}{}{}#1}\else
  \providecommand{\doi}{DOI~\discretionary{}{}{}\begingroup
  \urlstyle{rm}\Url}\fi

\bibitem{agueh2011barycenters}
Agueh, M., Carlier, G.: Barycenters in the wasserstein space.
\newblock SIAM Journal on Mathematical Analysis \textbf{43}(2), 904--924 (2011)

\bibitem{anikin2017dual}
Anikin, A.S., Gasnikov, A.V., Dvurechensky, P.E., Tyurin, A.I., Chernov, A.V.:
  Dual approaches to the minimization of strongly convex functionals with a
  simple structure under affine constraints.
\newblock Computational Mathematics and Mathematical Physics \textbf{57}(8),
  1262--1276 (2017)

\bibitem{antonakopoulos2019adaptive}
Antonakopoulos, K., Belmega, V., Mertikopoulos, P.: An adaptive mirror-prox
  method for variational inequalities with singular operators.
\newblock In: H.~Wallach, H.~Larochelle, A.~Beygelzimer, F.~d\textquotesingle
  Alch\'{e}-Buc, E.~Fox, R.~Garnett (eds.) Advances in Neural Information
  Processing Systems 32, pp. 8455--8465. Curran Associates, Inc. (2019).
\newblock
  \urlprefix\url{http://papers.nips.cc/paper/9053-an-adaptive-mirror-prox-method-for-variational-inequalities-with-singular-operators.pdf}

\bibitem{bach2019universal}
Bach, F., Levy, K.Y.: A universal algorithm for variational inequalities
  adaptive to smoothness and noise.
\newblock In: A.~Beygelzimer, D.~Hsu (eds.) Proceedings of the Thirty-Second
  Conference on Learning Theory, \emph{Proceedings of Machine Learning
  Research}, vol.~99, pp. 164--194. PMLR, Phoenix, USA (2019).
\newblock \urlprefix\url{http://proceedings.mlr.press/v99/bach19a.html}.
\newblock ArXiv:1902.01637

\bibitem{bayandina2018mirror}
Bayandina, A., Dvurechensky, P., Gasnikov, A., Stonyakin, F., Titov, A.: Mirror
  descent and convex optimization problems with non-smooth inequality
  constraints.
\newblock In: P.~Giselsson, A.~Rantzer (eds.) Large-Scale and Distributed
  Optimization, chap.~8, pp. 181--215. Springer International Publishing
  (2018).
\newblock \doi{10.1007/978-3-319-97478-1_8}.
\newblock ArXiv:1710.06612

\bibitem{benamou2015iterative}
Benamou, J.D., Carlier, G., Cuturi, M., Nenna, L., Peyré, G.: Iterative
  bregman projections for regularized transportation problems.
\newblock SIAM Journal on Scientific Computing \textbf{37}(2), A1111--A1138
  (2015)

\bibitem{beznosikov2021decentralized}
Beznosikov, A., Dvurechensky, P., Koloskova, A., Samokhina, V., Stich, S.U.,
  Gasnikov, A.: Decentralized local stochastic extra-gradient for variational
  inequalities.
\newblock arXiv:2106.08315  (2021)

\bibitem{bogachev2020real}
Bogachev, V.I., Smolyanov, O.G.: Real and functional analysis, vol.~4.
\newblock Springer (2020)

\bibitem{boissard2015distribution}
Boissard, E., Le~Gouic, T., Loubes, J.M.: Distribution’s template estimate
  with wasserstein metrics.
\newblock Bernoulli \textbf{21}(2), 740--759 (2015).
\newblock \doi{10.3150/13-BEJ585}.
\newblock \urlprefix\url{https://doi.org/10.3150/13-BEJ585}

\bibitem{bubeck2014convex}
Bubeck, S.: Convex optimization: Algorithms and complexity.
\newblock Foundations and Trends® in Machine Learning \textbf{8}(3-4),
  231--357 (2015).
\newblock \doi{10.1561/2200000050}.
\newblock \urlprefix\url{http://dx.doi.org/10.1561/2200000050}

\bibitem{chernov2016fast}
Chernov, A., Dvurechensky, P., Gasnikov, A.: Fast primal-dual gradient method
  for strongly convex minimization problems with linear constraints.
\newblock In: Y.~Kochetov, M.~Khachay, V.~Beresnev, E.~Nurminski, P.~Pardalos
  (eds.) Discrete Optimization and Operations Research: 9th International
  Conference, DOOR 2016, Vladivostok, Russia, September 19-23, 2016,
  Proceedings, pp. 391--403. Springer International Publishing (2016)

\bibitem{chewi2020gradient}
Chewi, S., Maunu, T., Rigollet, P., Stromme, A.: Gradient descent algorithms
  for {B}ures-{W}asserstein barycenters.
\newblock In: J.~Abernethy, S.~Agarwal (eds.) Proceedings of Thirty Third
  Conference on Learning Theory, \emph{Proceedings of Machine Learning
  Research}, vol. 125, pp. 1276--1304. PMLR (2020).
\newblock \urlprefix\url{http://proceedings.mlr.press/v125/chewi20a.html}

\bibitem{claici2018stochastic}
Claici, S., Chien, E., Solomon, J.: Stochastic {W}asserstein barycenters.
\newblock In: J.~Dy, A.~Krause (eds.) Proceedings of the 35th International
  Conference on Machine Learning, \emph{Proceedings of Machine Learning
  Research}, vol.~80, pp. 999--1008. PMLR (2018).
\newblock \urlprefix\url{http://proceedings.mlr.press/v80/claici18a.html}

\bibitem{cuturi2013sinkhorn}
Cuturi, M.: Sinkhorn distances: Lightspeed computation of optimal transport.
\newblock In: C.J.C. Burges, L.~Bottou, M.~Welling, Z.~Ghahramani, K.Q.
  Weinberger (eds.) Advances in Neural Information Processing Systems 26, pp.
  2292--2300. Curran Associates, Inc. (2013)

\bibitem{cuturi2014fast}
Cuturi, M., Doucet, A.: Fast computation of wasserstein barycenters.
\newblock In: E.P. Xing, T.~Jebara (eds.) Proceedings of the 31st International
  Conference on Machine Learning, \emph{Proceedings of Machine Learning
  Research}, vol.~32, pp. 685--693. PMLR, Bejing, China (2014).
\newblock \urlprefix\url{http://proceedings.mlr.press/v32/cuturi14.html}

\bibitem{delon2020wassersteintype}
Delon, J., Desolneux, A.: A wasserstein-type distance in the space of gaussian
  mixture models.
\newblock {SIAM} J. Imaging Sci. \textbf{13}(2), 936--970 (2020).
\newblock \doi{10.1137/19M1301047}.
\newblock \urlprefix\url{https://doi.org/10.1137/19M1301047}

\bibitem{dvinskikh2020stochastic}
Dvinskikh, D.: Stochastic averaging versus sample average approximation for
  population wasserstein barycenter calculation.
\newblock arXiv:2001.07697  (2020)

\bibitem{dvinskikh2021decentralized}
Dvinskikh, D.: Decentralized algorithms for wasserstein barycenters.
\newblock arXiv:2105.01587  (2021)

\bibitem{dvinskikh2019primal}
Dvinskikh, D., Gorbunov, E., Gasnikov, A., Dvurechensky, P., Uribe, C.A.: On
  primal and dual approaches for distributed stochastic convex optimization
  over networks.
\newblock In: 2019 IEEE 58th Conference on Decision and Control (CDC), pp.
  7435--7440. IEEE (2019).
\newblock ArXiv:1903.09844

\bibitem{dvinskikh2019adaptive}
Dvinskikh, D., Ogaltsov, A., Dvurechensky, P., Gasnikov, A., Spokoiny, V.:
  Adaptive gradient descent for convex and non-convex stochastic optimization.
\newblock arXiv preprint arXiv:1911.08380  (2019)

\bibitem{dvinskikh2020line-search}
Dvinskikh, D., Ogaltsov, A., Gasnikov, A., Dvurechensky, P., Spokoiny, V.: On
  the line-search gradient methods for stochastic optimization.
\newblock IFAC-PapersOnLine \textbf{53}(2), 1715--1720 (2020).
\newblock \doi{https://doi.org/10.1016/j.ifacol.2020.12.2284}.
\newblock
  \urlprefix\url{https://www.sciencedirect.com/science/article/pii/S240589632032944X}.
\newblock 21th IFAC World Congress, arXiv:1911.08380

\bibitem{dvinskikh2020improved}
Dvinskikh, D., Tiapkin, D.: Improved complexity bounds in wasserstein
  barycenter problem.
\newblock In: A.~Banerjee, K.~Fukumizu (eds.) Proceedings of The 24th
  International Conference on Artificial Intelligence and Statistics,
  \emph{Proceedings of Machine Learning Research}, vol. 130, pp. 1738--1746.
  PMLR (2021).
\newblock \urlprefix\url{http://proceedings.mlr.press/v130/dvinskikh21a.html}

\bibitem{dvurechensky2018decentralize}
Dvurechensky, P., Dvinskikh, D., Gasnikov, A., Uribe, C.A., Nedi\'c, A.:
  Decentralize and randomize: Faster algorithm for {W}asserstein barycenters.
\newblock In: S.~Bengio, H.~Wallach, H.~Larochelle, K.~Grauman,
  N.~Cesa-Bianchi, R.~Garnett (eds.) Advances in Neural Information Processing
  Systems 31, NeurIPS 2018, pp. 10783--10793. Curran Associates, Inc. (2018).
\newblock
  \urlprefix\url{http://papers.nips.cc/paper/8274-decentralize-and-randomize-faster-algorithm-for-wasserstein-barycenters.pdf}.
\newblock ArXiv:1806.03915

\bibitem{dvurechensky2016primal-dual}
Dvurechensky, P., Gasnikov, A., Gasnikova, E., Matsievsky, S., Rodomanov, A.,
  Usik, I.: Primal-dual method for searching equilibrium in hierarchical
  congestion population games.
\newblock In: Supplementary Proceedings of the 9th International Conference on
  Discrete Optimization and Operations Research and Scientific School (DOOR
  2016) Vladivostok, Russia, September 19 - 23, 2016, pp. 584--595 (2016).
\newblock ArXiv:1606.08988

\bibitem{dvurechensky2018computational}
Dvurechensky, P., Gasnikov, A., Kroshnin, A.: Computational optimal transport:
  Complexity by accelerated gradient descent is better than by {S}inkhorn's
  algorithm.
\newblock In: J.~Dy, A.~Krause (eds.) Proceedings of the 35th International
  Conference on Machine Learning, \emph{Proceedings of Machine Learning
  Research}, vol.~80, pp. 1367--1376 (2018).
\newblock ArXiv:1802.04367

\bibitem{dvurechensky2020stable}
Dvurechensky, P., Gasnikov, A., Omelchenko, S., Tiurin, A.: A stable
  alternative to {S}inkhorn's algorithm for regularized optimal transport.
\newblock In: A.~Kononov, M.~Khachay, V.A. Kalyagin, P.~Pardalos (eds.)
  Mathematical Optimization Theory and Operations Research, pp. 406--423.
  Springer International Publishing, Cham (2020)

\bibitem{dvurechensky2021hyperfast}
Dvurechensky, P., Kamzolov, D., Lukashevich, A., Lee, S., Ordentlich, E.,
  Uribe, C.A., Gasnikov, A.: Hyperfast second-order local solvers for efficient
  statistically preconditioned distributed optimization.
\newblock arXiv:2102.08246  (2021)

\bibitem{genevay2016stochastic}
Genevay, A., Cuturi, M., Peyr{\'e}, G., Bach, F.: Stochastic optimization for
  large-scale optimal transport.
\newblock In: Advances in neural information processing systems, pp. 3440--3448
  (2016)

\bibitem{gorbunov2019optimal}
Gorbunov, E., Dvinskikh, D., Gasnikov, A.: Optimal decentralized distributed
  algorithms for stochastic convex optimization.
\newblock arXiv preprint arXiv:1911.07363  (2019)

\bibitem{gorbunov2020recent}
Gorbunov, E., Rogozin, A., Beznosikov, A., Dvinskikh, D., Gasnikov, A.: Recent
  theoretical advances in decentralized distributed convex optimization.
\newblock arXiv preprint arXiv:2011.13259  (2020)

\bibitem{guminov2021combination}
Guminov, S., Dvurechensky, P., Tupitsa, N., Gasnikov, A.: On a combination of
  alternating minimization and {N}esterov's momentum.
\newblock In: M.~Meila, T.~Zhang (eds.) Proceedings of the 38th International
  Conference on Machine Learning, \emph{Proceedings of Machine Learning
  Research}, vol. 139, pp. 3886--3898. PMLR, Virtual (2021).
\newblock \urlprefix\url{http://proceedings.mlr.press/v139/guminov21a.html}.
\newblock ArXiv:1906.03622, WIAS Preprint No. 2695

\bibitem{guminov2019accelerated}
Guminov, S.V., Nesterov, Y.E., Dvurechensky, P.E., Gasnikov, A.V.: Accelerated
  primal-dual gradient descent with linesearch for convex, nonconvex, and
  nonsmooth optimization problems.
\newblock Doklady Mathematics \textbf{99}(2), 125--128 (2019)

\bibitem{heinemann2020randomised}
Heinemann, F., Munk, A., Zemel, Y.: Randomised wasserstein barycenter
  computation: Resampling with statistical guarantees.
\newblock arXiv preprint arXiv:2012.06397  (2020)

\bibitem{hendrikx2020optimal}
Hendrikx, H., Bach, F., Massoulie, L.: An optimal algorithm for decentralized
  finite sum optimization.
\newblock arXiv preprint arXiv:2005.10675  (2020)

\bibitem{hendrikx2020statistically}
Hendrikx, H., Xiao, L., Bubeck, S., Bach, F., Massoulie, L.: Statistically
  preconditioned accelerated gradient method for distributed optimization.
\newblock In: H.D. III, A.~Singh (eds.) Proceedings of the 37th International
  Conference on Machine Learning, \emph{Proceedings of Machine Learning
  Research}, vol. 119, pp. 4203--4227. PMLR (2020).
\newblock \urlprefix\url{http://proceedings.mlr.press/v119/hendrikx20a.html}

\bibitem{krawtschenko2020distributed}
Krawtschenko, R., Uribe, C.A., Gasnikov, A., Dvurechensky, P.: Distributed
  optimization with quantization for computing wasserstein barycenters.
\newblock arXiv:2010.14325  (2020).
\newblock \doi{10.20347/WIAS.PREPRINT.2782}.
\newblock WIAS preprint 2782

\bibitem{kroshnin2019statistical}
Kroshnin, A., Spokoiny, V., Suvorikova, A.: {Statistical inference for
  Bures–Wasserstein barycenters}.
\newblock The Annals of Applied Probability \textbf{31}(3), 1264 -- 1298
  (2021).
\newblock \doi{10.1214/20-AAP1618}.
\newblock \urlprefix\url{https://doi.org/10.1214/20-AAP1618}

\bibitem{kroshnin2019complexity}
Kroshnin, A., Tupitsa, N., Dvinskikh, D., Dvurechensky, P., Gasnikov, A.,
  Uribe, C.: On the complexity of approximating {W}asserstein barycenters.
\newblock In: K.~Chaudhuri, R.~Salakhutdinov (eds.) Proceedings of the 36th
  International Conference on Machine Learning, \emph{Proceedings of Machine
  Learning Research}, vol.~97, pp. 3530--3540. PMLR, Long Beach, California,
  USA (2019).
\newblock ArXiv:1901.08686

\bibitem{lafferty2005diffusion}
Lafferty, J., Lebanon, G.: Diffusion kernels on statistical manifolds.
\newblock Journal of Machine Learning Research \textbf{6}(Jan), 129--163 (2005)

\bibitem{lin2020revisiting}
Lin, T., Ho, N., Chen, X., Cuturi, M., Jordan, M.I.: Revisiting fixed support
  wasserstein barycenter: Computational hardness and efficient algorithms.
\newblock arXiv preprint arXiv:2002.04783  (2020)

\bibitem{lin2019efficient}
Lin, T., Ho, N., Jordan, M.: On efficient optimal transport: An analysis of
  greedy and accelerated mirror descent algorithms.
\newblock In: K.~Chaudhuri, R.~Salakhutdinov (eds.) Proceedings of the 36th
  International Conference on Machine Learning, \emph{Proceedings of Machine
  Learning Research}, vol.~97, pp. 3982--3991. PMLR, Long Beach, California,
  USA (2019)

\bibitem{mensch2020online}
Mensch, A., Peyr{\'e}, G.: Online sinkhorn: Optimal transport distances from
  sample streams.
\newblock Advances in Neural Information Processing Systems \textbf{33} (2020)

\bibitem{minh2010}
Minh, H.Q.: Nonparametric stochastic approximation with large step-sizes.
\newblock Some Properties of Gaussian Reproducing Kernel Hilbert Spaces and
  Their Implications for Function Approximation and Learning Theory
  \textbf{44}, 307--338 (2010).
\newblock \doi{10.1007/s00365-009-9080-0}.
\newblock \urlprefix\url{https://doi.org/10.1007/s00365-009-9080-0}

\bibitem{mohri2018foundations}
Mohri, M., Rostamizadeh, A., Talwalkar, A.: Foundations of machine learning.
\newblock MIT press (2018)

\bibitem{nemirovski2009robust}
Nemirovski, A., Juditsky, A., Lan, G., Shapiro, A.: Robust stochastic
  approximation approach to stochastic programming.
\newblock SIAM Journal on Optimization \textbf{19}(4), 1574--1609 (2009).
\newblock \doi{10.1137/070704277}.
\newblock \urlprefix\url{https://doi.org/10.1137/070704277}

\bibitem{nemirovsky1983problem}
Nemirovsky, A., Yudin, D.: Problem Complexity and Method Efficiency in
  Optimization.
\newblock J. Wiley \& Sons, New York (1983)

\bibitem{nesterov2020primal-dual}
Nesterov, Y., Gasnikov, A., Guminov, S., Dvurechensky, P.: Primal-dual
  accelerated gradient methods with small-dimensional relaxation oracle.
\newblock Optimization Methods and Software \textbf{0}(0), 1--28 (2020).
\newblock \doi{10.1080/10556788.2020.1731747}.
\newblock \urlprefix\url{https://doi.org/10.1080/10556788.2020.1731747}.
\newblock Accepted,arXiv:1809.05895

\bibitem{peyre2018optimaltransport}
Peyr\'e, G., Cuturi, M.: Computational optimal transport.
\newblock Foundations and Trends in Machine Learning \textbf{11}(5-6), 355--607
  (2019)

\bibitem{rockafellar2009va}
Rockafellar, R.T., Wets, R.J.B.: Variational analysis, vol. 317.
\newblock Springer Science \& Business Media (2009)

\bibitem{rogozin2021decentralized}
Rogozin, A., Beznosikov, A., Dvinskikh, D., Kovalev, D., Dvurechensky, P.,
  Gasnikov, A.: Decentralized distributed optimization for saddle point
  problems.
\newblock arXiv:2102.07758  (2021)

\bibitem{rogozin2021accelerated}
Rogozin, A., Bochko, M., Dvurechensky, P., Gasnikov, A., Lukoshkin, V.: An
  accelerated method for decentralized distributed stochastic optimization over
  time-varying graphs.
\newblock arXiv:2103.15598  (2021)

\bibitem{scaman2017optimal}
Scaman, K., Bach, F., Bubeck, S., Lee, Y.T., Massouli{\'e}, L.: Optimal
  algorithms for smooth and strongly convex distributed optimization in
  networks.
\newblock In: D.~Precup, Y.W. Teh (eds.) Proceedings of the 34th International
  Conference on Machine Learning, \emph{Proceedings of Machine Learning
  Research}, vol.~70, pp. 3027--3036. PMLR, International Convention Centre,
  Sydney, Australia (2017).
\newblock \urlprefix\url{http://proceedings.mlr.press/v70/scaman17a.html}

\bibitem{schmitzer2019stabilized}
Schmitzer, B.: Stabilized sparse scaling algorithms for entropy regularized
  transport problems.
\newblock SIAM Journal on Scientific Computing \textbf{41}(3), A1443--A1481
  (2019).
\newblock ArXiv:1610.06519

\bibitem{staib2017parallel}
Staib, M., Claici, S., Solomon, J.M., Jegelka, S.: Parallel streaming
  wasserstein barycenters.
\newblock In: I.~Guyon, U.V. Luxburg, S.~Bengio, H.~Wallach, R.~Fergus,
  S.~Vishwanathan, R.~Garnett (eds.) Advances in Neural Information Processing
  Systems 30, pp. 2647--2658. Curran Associates, Inc. (2017).
\newblock
  \urlprefix\url{http://papers.nips.cc/paper/6858-parallel-streaming-wasserstein-barycenters.pdf}

\bibitem{steinwart2008support}
Steinwart, I., Christmann, A.: Support vector machines.
\newblock Springer Science \& Business Media (2008)

\bibitem{stonyakin2018generalized}
Stonyakin, F., Gasnikov, A., Dvurechensky, P., Alkousa, M., Titov, A.:
  Generalized {M}irror {P}rox for monotone variational inequalities:
  Universality and inexact oracle.
\newblock arXiv:1806.05140  (2018)

\bibitem{stonyakin2021inexact}
Stonyakin, F., Tyurin, A., Gasnikov, A., Dvurechensky, P., Agafonov, A.,
  Dvinskikh, D., Alkousa, M., Pasechnyuk, D., Artamonov, S., Piskunova, V.:
  Inexact model: A framework for optimization and variational inequalities.
\newblock Optimization Methods and Software  (2021).
\newblock \doi{10.1080/10556788.2021.1924714}.
\newblock \urlprefix\url{https://doi.org/10.1080/10556788.2021.1924714}.
\newblock WIAS Preprint No. 2709, arXiv:2001.09013, arXiv:1902.00990

\bibitem{stonyakin2019gradient}
Stonyakin, F.S., Dvinskikh, D., Dvurechensky, P., Kroshnin, A., Kuznetsova, O.,
  Agafonov, A., Gasnikov, A., Tyurin, A., Uribe, C.A., Pasechnyuk, D.,
  Artamonov, S.: Gradient methods for problems with inexact model of the
  objective.
\newblock In: M.~Khachay, Y.~Kochetov, P.~Pardalos (eds.) Mathematical
  Optimization Theory and Operations Research, pp. 97--114. Springer
  International Publishing, Cham (2019).
\newblock ArXiv:1902.09001

\bibitem{uribe2018distributed}
{Uribe}, C.A., {Dvinskikh}, D., {Dvurechensky}, P., {Gasnikov}, A., {Nedi\'c},
  A.: Distributed computation of {W}asserstein barycenters over networks.
\newblock In: 2018 IEEE Conference on Decision and Control (CDC), pp.
  6544--6549 (2018).
\newblock ArXiv:1803.02933

\end{thebibliography}

%
%

\end{document}